\documentclass[11pt]{amsart}

\usepackage{amsmath}
\usepackage{amsthm}
\usepackage{stmaryrd}
\usepackage{amsfonts,mathrsfs,amssymb}%,txfonts}
\usepackage{times}
\usepackage{fullpage}
\title{Bergman interpolation on finite Riemann surfaces.\\
Part I:  Asymptotically Flat Case}
\author{Dror Varolin} 
\email{dror@math.sunysb.edu}
\address{Department of Mathematics \newline \indent Stony BrookUniversity \newline \indent Stony Brook, NY 11794-3651}

\newcommand{\noi}{\noindent}

\newcommand{\co}{{\mathcal O}}

\newcommand{\sa}{{\mathscr A}}

\newcommand{\sC}{{\mathscr C}}

\newcommand{\se}{{\mathscr E}}

\newcommand{\sh}{{\mathscr H}}

\newcommand{\sr}{{\mathscr R}}

\newcommand{\fp}{{\mathfrak p}}

\newcommand{\vp}{\varphi} 
\newcommand{\ve}{\varepsilon}

\newcommand{\A}{{\mathbb A}}

\newcommand{\C}{{\mathbb C}}
\newcommand{\D}{{\mathbb D}}

\newcommand{\N}{{\mathbb N}}
\newcommand{\p}{{\mathbb P}}

\newcommand{\R}{{\mathbb R}}

\newcommand{\red}{\hfill $\diamond$}

\newcommand{\di}{\partial}
\newcommand{\dbar}{\bar \partial}

\newcommand{\re}{{\rm Re\ }}
\newcommand{\im}{{\rm Im\ }}

\newcommand{\relcomp}{\subset \subset}

\newcommand{\ii}{\sqrt{-1}}

\newcommand{\tensor}{\otimes}

\def\XXint#1#2#3{{\setbox0=\hbox{$#1{#2#3}{\int}$} 
\vcenter{\hbox{$#2#3$}}\kern-.5\wd0}}

\usepackage{hyperref}

\begin{document}
%\setlength{\textwidth}{24cm}
%\setlength{\textheight}{30cm}
%\maketitle

\theoremstyle{plain}
\newtheorem{thm}{\sc Theorem}
\newtheorem*{s-thm}{\sc Theorem}
\newtheorem{lem}{\sc Lemma}[section]
\newtheorem{d-thm}[lem]{\sc Theorem}
\newtheorem{prop}[lem]{\sc Proposition}
\newtheorem{cor}[lem]{\sc Corollary}

\theoremstyle{definition}
\newtheorem{conj}[lem]{\sc Conjecture}
\newtheorem{prob}[lem]{\sc Open Problem}
\newtheorem{defn}[lem]{\sc Definition}
\newtheorem*{s-defn}{\sc Definition}
\newtheorem{qn}[lem]{\sc Question}
\newtheorem{ex}[lem]{\sc Example}
\newtheorem{rmk}[lem]{\sc Remark}
\newtheorem*{s-rmk}{\sc Remark}
\newtheorem{rmks}[lem]{\sc Remarks}
\newtheorem*{ack}{\sc Acknowledgment}

%\begin{center}
%{\Large {\bf Restricting Weighted-Square-Integrable Holomorphic Functions to \\ 
%\vskip .08in

%Hypersurfaces in Complex Euclidean Space}}
%\\ \ 
%\\ \ 
%{\large {\sf Stanislav Ostrovsky \quad {\small {\rm and}} \quad Dror Varolin\footnote{Partially supported by an NSF grant}}}

%\end{center}

\maketitle

%\setcounter{tocdepth}2

%\tableofcontents

%\setcounter{section}{1}

\vskip .3in

%\begin{center}{\sc Abstract}\end{center}

%\noi {\small }

\begin{abstract}
We study the Bergman space interpolation problem of open Riemann surfaces obtained from a compact Riemann surface by removing a finite number of points. We equip such a surface with what we call an asymptotically flat conformal metric, i.e., a complete metric with zero curvature outside a compact subset. We then establish sufficient conditions for interpolation in weighted Bergman spaces over asymptotically flat Riemann surfaces.  When our weights have curvature that is quasi-isometric to the asymptotically flat boundary metric, we show that these sufficient conditions are necessary, unless the surface has at least one cylindrical end, in which case, the necessary conditions are slightly weaker than the sufficiency conditions.
\end{abstract}

\section*{Introduction}

A fundamental topic in complex analysis is the so-called interpolation problem for Bergman spaces.  To describe the problem, let $X$ be an open Riemann surface with conformal metric $\omega$, let $\psi : X \to [-\infty, \infty)$ be a weight function on $X$, and let $\Gamma \subset X$ a closed discrete subset.   
\begin{enumerate}
\item[(a)] We define the Hilbert spaces 
\[
\sh ^2 (X, e^{-\psi}\omega) := \left \{ g \in \co (X)\ ;\ \int _X |g|^2 e^{-\psi} \omega < +\infty \right \}
\]
and 
\[
\ell ^2 (\Gamma , e^{-\psi}) := \left \{ f :\Gamma \to \C \ ;\ \sum _{\gamma \in \Gamma} |f(\gamma)|^2e^{-\psi(\gamma)} < +\infty \right \}.
\]
\item[(b)] We say that $\Gamma$ is an interpolation sequence (for the triple $(X,\omega,\psi)$) if the restriction map 
\[
\sr _{\Gamma} : \sh ^2(X,e^{-\psi}\omega) \to \ell ^2 (\Gamma, e^{-\psi})
\]
is surjective, i.e., for any $f \in \ell ^2 (\Gamma ,e^{-\psi})$ there exists $F \in \sh ^2 (X,e^{-\psi}\omega)$ such that $F|_{\Gamma} = f$.
\red
\end{enumerate}

\noi Given a triple $(X, \omega , \psi)$, a complete solution of the interpolation problem consists in characterizing interpolation sequences $\Gamma \subset X$ among all closed discrete subsets of $X$.  Preferably, one characterizes such $\Gamma$ by geometric properties expressed in terms of the metric $\omega$ and the weight $\psi$.

\begin{s-rmk}
In their paper \cite{ss}, Shapiro and Shields defined a general interpolation problem for Hilbert spaces of holomorphic functions.  In Section \ref{SSI} we will recall the Shapiro-Shields interpolation problem, and we will show that, in the cases we consider here, the two interpolation problems are identical.
\red
\end{s-rmk}

\begin{s-rmk}
There is also a companion {\it sampling problem} for Bergman spaces, that examines the injectivity of the restriction map (and requires the boundedness of the inverse).  Though it is an interesting and important problem, the solution of the sampling problem involves different methods, and will not be considered in the present article.
\red
\end{s-rmk}

We study the interpolation problem in the Bergman space of an open Riemann surface that is obtained from a compact Riemann surface by removing a finite number of points.  Although such surfaces have a canonical metric of constant curvature (with this curvature equal to zero when the surface is $\p _1$ with one or two points removed, and negative otherwise), we are going to consider metrics that are, in general, slightly less canonical.  Namely, our metrics are asymptotically flat, but even more restricted.  More precisely, if we have an open Riemann surface $X$ that is the complement of finitely many points in a compact Riemann surface $Y$, we can find a compact set $K \relcomp X$ with smooth, $1$-dimensional boundary, such that the complement of $K$ is a finite number of disjoint sets $U_1,...,U_N$ with each $U_j$ is biholomorphic to the punctured disk $\D^* := \D - \{0\}$.  We assume that $X$ is equipped with a smooth conformal metric $\omega$ (which we think of as a positive $(1,1)$-form) such that for each $j$, $\omega |_{U_j}$ is holomorphically isometric to a constant multiple of one of the following two metrics on $\D ^*$:
\begin{enumerate}
\item[(i)] The inverted Euclidean metric 
\[
\omega _{o} := \frac{\ii dz \wedge d\bar z}{2 |z|^4}.
\]
\item[(ii)] The cylindrical metric 
\[
\omega _c := \frac{\ii dz \wedge d\bar z}{2|z|^2}.
\]
\end{enumerate}

\begin{s-rmk}
There are other flat metrics on the punctured disk which are not the inverted Euclidean or cylindrical metric.  In fact, all these flat metrics are isometric to the metrics $e^{\re F}\omega_{\alpha}$ for some $f \in \co (\D)$ and $\alpha \in \R$, where
\[
\omega _{\alpha}:= \frac{\ii dz \wedge d\bar z}{2|z|^{2\alpha}}.
\]
The puncture is infinitely far away in these metrics if and only if $\alpha \ge 1$.  So the critical case is the cylindrical metric, and all other cases are cones that open.  The angle of the cone is greater than 90 degrees as soon as $\alpha > 2$.

The analysis of these more general cones is more delicate.  The metrics have been treated, in a certain sense, in the work \cite{mmo} of Marco, Massaneda and Ortega Cerd\`a, which actually treats metrics that not necessarily flat.  However, the results of \cite{mmo} do not apply to the cylindrical case.  Nevertheless, the results of \cite{mmo} can be adapted to extend the results of the present paper to general asymptotically flat metrics.  In fact, any non-cylindrical flat end can be treated, using the results of \cite{mmo}, in a way analogous to the way we treated Euclidean ends here.  In the interest of brevity, we have restricted ourselves to the case in which only cylindrical and Euclidean ends are present.
\red
\end{s-rmk}

\begin{s-rmk}
As we just mentioned, there is another possibility for a metric of constant curvature, with the curvature being negative, but this case needs to be treated differently given the current state of the art of $L^2$ methods, particularly regarding $L^2$ extension.  We therefore consider the negatively curved case in the sequel \cite{v-rs2} to the present article.
\red
\end{s-rmk}

The main result of this paper is the following theorem.

\begin{thm}\label{main}
Let $X$ a Riemann surface obtained from a compact Riemann surface by removing a finite number of points, and let $\omega$ be an asymptotically flat conformal metric on $X$.  Let $\vp \in \sC ^2(X)$ be a smooth weight function, and assume there exist positive constants $m < M$ such that 
\begin{equation}\label{curv-bd-gen}
m \omega \le \ii \di \dbar \vp + {\rm R}(\omega) \le M \omega,
\end{equation}
where ${\rm R}(\omega)$ is the curvature $(1,1)$-form of $\omega$.  Let $\Gamma \subset X$ be a closed discrete subset.  Denote the restriction map by $\sr _{\Gamma} : \sh ^2 (X, e^{-\vp}\omega) \to \ell ^2 (\Gamma , e^{-\vp})$.  If\begin{enumerate}
\item[(i+)] $\Gamma$ is uniformly separated with respect to the geodesic distance associated to $\omega$, and 
\item[(ii+)] the asymptotic (upper) density $D^+_{\vp}(\Gamma)$ of $\Gamma$ is strictly less than $1$,
\end{enumerate}
then $\sr _{\Gamma}$ is surjective.  Conversely, if $\sr_{\Gamma}$ is surjective, then 
\begin{enumerate}
\item[(i-)] $\Gamma$ is uniformly separated with respect to the geodesic distance associated to $\omega$, and 
\item[(ii-)] $D^+_{\vp}(\Gamma) \le 1$.  Moreover, if none of the ends are cylindrical, then $D^+_{\vp}(\Gamma)<1$.
\end{enumerate}
\end{thm}

\noi (Here and in the rest of the paper, a sequence $\Gamma$ is said to be uniformly separated with respect to some distance function $\rho$ if the number $\inf \{ \rho (\gamma _1, \gamma _2)\ ;\ \gamma _1,\gamma _2 \in \Gamma,\ \gamma _1 \neq \gamma _2\}$ is positive.)

\begin{s-rmk}
The non-strictness of the density bound in (ii-) of Theorem \ref{main} when $X$ has at least one cylindrical end is not an artifact of the proof, but rather the best one can do.  Indeed, an example of Borichev and Lyubarskii \cite{bl} exhibits a sequence $\Gamma$ in the Riemann sphere with one cylindrical puncture, such that $\sr _{\Gamma}$ is surjective (in fact, in their case, it is bijective) and $D^+_{\vp}(\Gamma) = 1$.
\red
\end{s-rmk}

Roughly speaking, the asymptotic density of $\Gamma$ is the least upper bound of certain weighted densities of the number of points of $\Gamma$ in large geodesic disks, the least upper bound being taken over all possible centers of the disks.  We shall give the precise definition of the asymptotic density $D^+_{\vp}(\Gamma)$ later in the introduction. 

\medskip

The history of the interpolation problem for Bergman spaces is surprisingly not very old.  As we already mentioned, in \cite{ss} Shapiro and Shields introduced the problem of studying interpolation sequences in Bergman spaces.  The first characterization of interpolation sequences for Bergman spaces was achieved by Seip and Wallsten \cite{s1,sw} for the case of the classical Bargmann-Fock space $X=\C$, $\omega= \omega _o$, and $\psi (z) = |z|^2$.  In this case, it was shown that a sequence $\Gamma \subset \C$ is an interpolation sequence if and only if 
\begin{enumerate}
\item[(i)] $\Gamma$ is uniformly separated with respect to the Euclidean distance, and 
\item[(ii)] the asymptotic density of $\Gamma$ is below a very precise threshold; with the appropriate normalization, 
\[
D^+ (\Gamma):= \limsup _{r \to \infty} \sup _{z\in \C} \frac{\# (\Gamma \cap D_r^o(z))}{r^2} < 1.
\]
\end{enumerate}
Seip then established an analogous result for the Bergman space in the unit disk \cite{s2}, which we will not state precisely here.  Berndtsson and Ortega Cerd\`a generalized the sufficiency part of Seip's Theorems to much more general weights in $\C$ and in the unit disk.  We will not state their results for the unit disk here, but their interpolation theorem in $\C$ can be stated as follows. 

\begin{d-thm}\cite{quimbo}  \label{quimbo-thm}
Let $\vp \in \sC ^2(\C)$ satisfy $0< m \le \frac{\di ^2 \vp}{\di z \di\bar z} \le M$ for some constants $m$ and $M$.  If $\Gamma \subset \C$ is uniformly separated with respect to the Euclidean distance, and if 
\[
D^+_{\vp} (\Gamma) := \limsup _{r \to \infty} \sup _{z \in \C} \frac{\# (\Gamma \cap D^o_r(z))}{\frac{1}{\pi} \int _{D^o_r(z)} \Delta \vp} <1,
\]
then $\Gamma$ is an interpolation set.
\end{d-thm}

The converse of Theorem \ref{quimbo-thm} in the entire plane was proved by Ortega Cerd\`a and Seip \cite{quimseep}, who also indicated how one can establish necessity for the case of the unit disk. 

Recently, Pingali and the author \cite{pv} established an improvement of Theorem \ref{quimbo-thm} in which arbitrary (pluri)subharmonic weights satisfying a density condition are allowed.  The article \cite{pv} concerns the higher dimensional version of the interpolation problem, and makes use of an Ohsawa-Takegoshi type extension theorem stated below as Theorem \ref{ot-basic}.  The interpolation theorem in dimension $1$ is a little easier to prove, and is established below as Theorem \ref{s-suff-eucl} (in a slightly different form than that of \cite{pv}).

The interpolation problem for more general open Riemann surfaces was first considered by Schuster and the author \cite{sv1}.  That article gave very general sufficient conditions for interpolation (and sampling) on finite (and a few other) Riemann surfaces, but it was not expected that all of these conditions would also be necessary.  Later, Ortega Cerd\`a \cite{quim-rs} considered interpolation and sampling problems for finite Riemann surfaces with only codimension-$1$ boundary.  He gave necessary and sufficient conditions for interpolation and sampling for $L^p$ analogs of our Hilbert spaces, for $1 \le p \le \infty$.  We will discuss Ortega Cerd\`a's Theorem in \cite{v-rs2}.  More importantly for us, in \cite{quim-rs} Ortega Cerd\`a made the crucial observation that the asymptotic density of a sequence is completely determined by the behavior of that sequence near the boundary of the surface; an idea that we will make extensive use of here.

Ortega Cerd\`a did not allow punctures, i.e., $0$-dimensional boundary components, for the Riemann surfaces he considered.  To some extent, the present article and its forthcoming sequel grew out of a desire to understand interpolation problems in the presence of punctures.  

Let us now turn to our definition of the asymptotic density. We will first define the asymptotic density in two special cases, namely the Euclidean case $(\C, \omega _o)$\footnote{Here we use the notation $\omega_o = \frac{\ii}{2}dz \wedge d\bar z$, even though we already used the same notation for the inverted metric.}  and the cylindrical case $(\C^*, \omega _c)$ (see (ii) above for the definition of $\omega _c$), and then give the general definition for asymptotically flat finite Riemann surfaces.

\begin{enumerate}
\item[(a)] \underline{Euclidean case}:  Given a closed discrete subset $\Gamma \subset \C$, we can find a function $T \in \co (\C)$ such that 
\[
{\rm Ord}(T) = \Gamma.
\]
Here and below, ${\rm Ord}$ denotes the order divisor, i.e., ${\rm Ord}(T)$ is a divisor supported on the zero set of $T$, and the integer assigned to each $z \in T^{-1}(0)$ is the order of vanishing of $T$ at $z$.  Thus saying that ${\rm Ord}(T) = \Gamma$ means that $T$ vanishes to order $1$ at each point of $\Gamma$, and has no other zeros.  For a given radius $r > 0$, we can define the {\it logarithmic average} of $\log |T|^2$ over the Euclidean annulus $\A^o_r(z)$ of inner radius $1$ and outer radius $r$, and center $z \in \C$, as 
\[
\lambda ^T_r (z) := \frac{1}{c_r} \int _{\A^o_r(z)}  \log |T(\zeta)|^2 \log \frac{r^2}{|\zeta -z|^2}  \omega _o (\zeta),
\]
where $c_r := \lambda _r ^{\sqrt{e}} = \pi (r^2 -1 +\log \frac{1}{r^2})$.  The function $\lambda ^T_r$ is subharmonic and locally bounded, and the distribution 
\[
\Upsilon ^{\Gamma}_r (z) := \ii \di \dbar \lambda ^T_r (z)
\]
is independent of the choice of $T$ satisfying ${\rm Ord}(T)=\Gamma$.  In fact, by the Poincar\`e-Lelong Formula, 
\[
\Upsilon ^{\Gamma}_r (z) =  \frac{2\pi}{c_r} \int _{\A^o _r(z)} \log \frac{r^2}{|\zeta -z|^2} \delta _{\Gamma},
\]
where $\delta _{\Gamma} := \sum _{\gamma \in \Gamma} \delta _{\gamma}$ is the sum of the point masses on the points of $\Gamma$.

\begin{defn}\label{Eucl-density-defn}
The asymptotic upper density of $\Gamma$ with respect to a subharmonic weight $\vp$ is the (possibly infinite) non-negative number 
\[
D^+_{\vp}(\Gamma) := \inf \left \{ \frac{1}{\alpha} \ ;\ \forall \ r_o > 0 \ \exists \ r > r_o\text{ such that }\ii \di \dbar \vp_r - \alpha \Upsilon ^{\Gamma}_r \ge 0\right \},
\]
where 
\begin{equation}\label{log-avg-defn}
\vp _r (z) := \frac{1}{\pi r^2} \int _{D^o_r(z)} \log \frac{r^2}{|\zeta -z|^2} \vp (\zeta) \omega _o(\zeta)
\end{equation}
is the logarithmic average of $\vp$ over the Euclidean disk of radius $r$ centered at $z$.
\red
\end{defn}

\begin{s-rmk}
Note that if the weight $\vp$ is sufficiently regular, then 
\[
D^+_{\vp} (\Gamma) = \limsup _{r \to \infty} \sup _{z \in \C} \frac{2\pi \int _{\A^o_r(z)} \log \frac{r^2}{|\zeta -z|^2} \delta _{\Gamma}(\zeta)}{\int _{D^o_r(z)} \log \frac{r^2}{|\zeta -z|^2}\Delta \vp (\zeta)},
\]
which is a logarithmic version of the asymptotic upper density in Theorem \ref{quimbo-thm}.  In fact, Ortega Cerd\`a and Seip pointed out that these two densities are equivalent.  
\red
\end{s-rmk}

\item[(b)]\underline{C}y\underline{lindrical case}:  For a number of reasons, it is convenient to work on the universal cover.  The exponential map $\fp : \C \to \C^*; \zeta \mapsto e^{\zeta}$ is the universal covering map of $\C^*$, and is an isometry of the Euclidean and cylindrical metrics.  

\begin{defn}\label{cyl-density-defn}
Given a closed discrete subset $\Gamma \subset \C^*$, we define the {\it cover density} of $\Gamma$ with respect to $\vp$ as 
\[
\tilde D^+_{\vp}( \Gamma) := D^+_{\tilde \vp}(\tilde \Gamma),
\]
where $\tilde \Gamma:= \fp ^{-1}(\Gamma)$ and $\tilde \vp:= \fp ^* \vp$.  
\red
\end{defn}

\item[(c)] \underline{General case}:  Now let $(X,\omega)$ be an asymptotically flat finite Riemann surface with either cylindrical or Euclidean ends, and denote by $U_1,...,U_N$ its asymptotically flat ends.  Each end $U_i$ comes with a biholomorphic map $F_i : \C - D^o_r(0) \to U_i$ of the complement of some Euclidean disk centered at $0$ to $U_i$, and $F_i$ is an isometry of $\omega$ and either the cylindrical or Euclidean metric.  If $\omega |_{U_i}$ is isometric under $F_i$ to the Euclidean metric,  we define 
\[
D^+_{\vp, i}(\Gamma) := D^+_{F_i^*\vp} (F_i ^{-1}(\Gamma \cap U_i)).
\]
And if $\omega |_{U_i}$ is isometric under $F_i$ to the cylindrical metric,  we define 
\[
D^+_{\vp, i}(\Gamma) := \tilde D^+_{F_i^*\vp} (F_i ^{-1}(\Gamma \cap U_i)).
\]

\begin{defn}\label{upper-density-gen-def}
The number 
\[
D^+_{\vp} (\Gamma) := \max_{1\le i \le n+m} D^+_{\vp, i}(\Gamma)
\]
is called the {\it asymptotic upper density} of $\Gamma \subset X$ with respect to the weight $\vp$.
\red
\end{defn}
\end{enumerate}

\begin{s-rmk}
In Definition \ref{upper-density-gen-def} we are glossing over one point:  the weight functions $F_i^*\vp$ are not defined on the whole plane or punctured plane, yet in the definition of density we average over large disks which might exit the domain of definition of these pulled back weights.  There is an easy way to remedy this problem: one cuts off the weights and adds a multiple of the Euclidean metric in the complement.  However, it is not even necessary to go to such pedantic lengths because, as we already mentioned, the density is completely determined by the "infinite tails" of the sequence.  In other words, if we threw away and finite subset of $\Gamma$, the resulting sequence would have the same density as $\Gamma$.  Said another way, we can restrict ourselves to averaging over large disks that lie in the domain of the weight $F_i^* \vp$.
\red
\end{s-rmk}

\begin{s-rmk}
It is not hard to show that when $X = \C$ or $X=\C^*$ with the Euclidean or cylindrical metric respectively, then the number $D^+_{\vp}(\Gamma)$ is the density or the cover density of $\Gamma$ respectively.
\red
\end{s-rmk}

The organization of the paper is as follows.  In Section \ref{background-section} we establish some basic background theory, most of it known and all of it essentially known.  In Section \ref{SSI} we recall the classical notion of interpolation sets in the sense of Shapiro-Shields, and show that, in the case of asymptotically flat Riemann surfaces with Euclidean and cylindrical ends, our notion of interpolation sequences agrees with the Shapiro-Shields notion, thus lending additional motivation to our definitions.  In Section \ref{Eucl-section} we prove Theorem \ref{main} for the special case $(X,\omega) = (\C, \omega _o)$.  The proof splits up into two parts.  In the first part we prove the sufficiency of the conditions of Theorem \ref{main} for interpolation, and in the second part we prove the necessity of these conditions for any interpolation sequence.  In fact, we prove a slightly stronger version of the main theorem, in which we weaken the lower bounds on the curvature of the weight $\vp$.  More importantly, we prove a stronger sufficiency result based on the $L^2$ Extension Theorem \ref{ot-basic}. The improved sufficiency theorem is very similar to work done by the author and Pingali \cite{pv}, and is just a slight modification of that work, including a simplification that arises in the $1$-dimensional setting.  Our proof of necessity follows closely the work of Ortega Cerd\`a and Seip \cite{quimseep}.  In Section \ref{cyl-section} we establish Theorem \ref{main} in the cylindrical case, with a similar strong sufficiency result.  Of course, here we have two ends, both of which are cylindrical, so we can only prove that the density of an interpolation set is at most $1$.  Finally in Section \ref{main-section} we finish the proof of Theorem \ref{main}.  Necessity is a relatively easy consequence of the two special cases, and sufficiency is handled in a manner similar to the special cases, except that we do not get quite as strong a sufficiency result in the general setting.

\begin{ack}
I am grateful to Henri Guenancia, Jeff McNeal, Quim Ortega Cerd\`a, Mihai P\u aun and Vamsi Pingali for many fruitful and stimulating discussions.  Finally, thanks are due to the referee for a careful reading of the manuscript and a number of very useful suggestions, including the suggestion that we compare our notion of interpolation sequences with the classical notion of Shapiro-Shields Interpolation for Hilbert spaces of holomorphic functions.
\end{ack}

\section{Background}\label{background-section}

Let $X$ be a Riemann surface.  We write  $d^c = \frac{\ii}{2} (\dbar - \di)$, and denote by
\[
\Delta := dd^c = \ii \di \dbar 
\]
the Laplace operator (so normalized).  

\subsection{Complete flat Hermitian metrics}

As is well known, every Riemann surface admits a complete Hermitian metric of constant curvature, i.e., a metric $\omega$ satisfying 
\[
\Delta \omega = c \omega
\]
for some constant $c$.  Once the surface is fixed, the sign of $c$ is determined.  If we further fix $c$ with the given sign, the metric is unique when $c$ is non-zero, and in the flat case it is determined by some kind of cohomology class.  

Only one Riemann surface has a complete positively curved conformal metric of constant curvature, namely $\p _1$.  Relatively few Riemann surfaces have a complete flat conformal metric:  these are $\C, \C^*$ and all complex tori.  All other Riemann surfaces have a complete metric of constant negative curvature, as they are covered by the disk.  

Let us look first at complete conformal metrics of identically zero curvature.  Since we are not interested in compact Riemann surfaces in this article, the only cases are $\C$ and $\C ^*$.  We shall refer to these as the Euclidean and cylindrical cases respectively.

\begin{enumerate}
\item[(i)] {\bf Euclidean case}: Of course, on $\C$ we have the Euclidean metric $g _o = |dz|^2$.  A result in Riemannian geometry says that if a complete Riemannian manifold has constant (sectional) curvature, then the exponential map exists on the entire tangent space and is a Riemannian covering map, with respect to the constant metric on $T_{\C , 0}$.  From this result it is not hard to show that any complete conformal metric $g$ on $\C$ is a constant multiple of $g_o$.  Indeed, let $g = e^{h}g_o$ be a conformal metric in $\C$ with $h(0)=0$ and let $F :T_{\C , 0} \to \C$ be the exponential map.  Since $\C$ is simply connected, $F$ is a diffeomorphism, and moreover it satisfies $F^*g = g_o$.  But 
\[
F^*(e^h |dz|^2) = e^{F^*h} |\di F +\dbar F |^2 = e^{F^*h}( |\di F|^2 + |\dbar  F|^2  +2\re \di F \overline{\di F}).
\]
Since the metric $g_o$ on $\C \cong T_{\C, 0}$ is conformal, we must have $\di F= 0$ or $\dbar F =0$.  Since the orientation of the tangent space is the same as that of the manifold, we must have the latter, so that $F$ is holomorphic.  It follows that $F \in {\rm Aut}(\C)$, and since $F$ preserves the origin, it must be a homothety, i.e., $g = ag_o$ for some positive constant $a$.

In the rest of the article, we denote by $\omega _o$ the (metric form of the) Euclidean metric.

\bigskip

\item[(ii)] {\bf Cylindrical case}:  On $\C ^*$ we have the complete flat metric 
\[
g_c := \frac{|d\zeta|^2}{|\zeta|^2}.
\]
(Note that this metric is invariant under the inversion $\zeta \mapsto \zeta ^{-1}$, so that the singularity is the same at $0$ and $\infty$.  The metric is also invariant under the scaling maps $\zeta \mapsto c \zeta$, $c \in \C^*$, and thus we have ${\rm Aut}(\C^*) \subset {\rm Isom}(\omega _o)$.)  If we take any holomorphic covering map $\fp : \C \to \C^*$ sending $0$ to $1$ (it is easy to see that then $\fp (z) = e^{a z}$ for some $a \in \C$) then
\[
\fp ^* g_c = \frac{|e^{a z} dz|^2}{|e^{a z}|^2} = |a|^2 |dz|^2
\]
is a constant multiple of the Euclidean metric.

Now let $g$ be any complete flat conformal metric on $\C ^*$, normalized so that $g(1) =g_c(1)$.  By the result of Riemannian geometry mentioned in (i), the exponential map $F : (T_{\C^*, 1}, g_o) \to (\C^*,g)$ is a Riemannian covering map.  Since the two metrics are conformal and $F$ is a local isometry and covering map, the same calculation as in (i) shows that $F$ must be holomorphic.  But then  
\[
F (z) = e^{a z},
\]
for some $a \in \C$, and so it follows that the metric $g$ is a constant multiple of $g_c$.

In the rest of the article, we denote by
\[
\omega _c = \frac{\ii dz \wedge d\bar z}{2|z|^2}
\]
(the K\"ahler form of) the cylindrical metric on $\C^*$.
\end{enumerate}

While the cylindrical and Euclidean metrics are the only complete flat metics, there are other flat metrics on the punctured disk that are "complete near the puncture", i.e., metrics $\omega$ on $\D^* \cup \di \D$ such that for each $z \in \D^*$ such that 
\[
\lim _{\zeta \to 0} d_{\omega}(z,\zeta) = +\infty.
\]
Let $\omega$ be such a metric.  We can write 
\[
\omega = f \omega _o,
\]
where $f$ is a positive function.  If $\omega$ is flat, then $\log f$ is harmonic.  In the punctured disk, any nowhere zero harmonic function is of the form $|z|^{-2\alpha}e^{\re F}$ for some $\alpha \in \R$ and some $F \in \co (\D)$. Indeed, 
\[
\int _{\di \D} d^c \log |z|^2 = \frac{1}{2\ii} \int _{\di \D} \frac{dz}{z} - \frac{d\bar z}{\bar z}  = 2\pi, 
\]
so if 
\[
\alpha := - \frac{1}{2\pi}\int _{\di \D} d^c \log f,
\]
then $\log f + \alpha \log |z|^2$ has no periods.  Therefore there is a holomorphic function $F \in \co (\D)$ such that 
\[
\log f = \log |z|^{-2\alpha} + 2\re F = \log (|z|^{2\alpha} e^{2F}).
\]

\begin{s-rmk}
From the point of view of this paper, one can assume all the end metrics are the flat metrics $\omega _{\alpha}$.  Indeed, if one starts with an end metric $e^{2\re F}\omega _{\alpha}$,  the factor $e^{2\re F}$ can be absorbed into the weight function.  The weights $\vp$ and $\vp +2\re F$ have the same curvature, and our hypotheses involve only the curvature of the weights.
\red
\end{s-rmk}

As one can easily check, completeness at the puncture implies that $\alpha \ge 1$.  As already mentioned, we will only consider the two cases $\alpha =1$ (cylinder) and $\alpha =2$ (Euclidean). 

\subsection{The $\mathbf{L^2}$ Extension Theorem}

In this section, we recall the following well-known result, which is often called an Ohsawa-Takegoshi type extension theorem, and which by now has many statements and proofs.  Here we state the version in \cite{v-tak}.

\begin{d-thm}\label{ot-basic}
Let $(X,\omega)$ be a Stein K\"ahler manifold of complex dimension $n$, and let $Z \subset X$ be a smooth hypersurface.  Assume there exists a section $T \in H^0(X,L_Z)$ and a metric $e^{-\lambda}$ for the line bundle $L_Z \to X$ associated to the smooth divisor $Z$, such that $e^{-\lambda}|_Z$ is still a singular Hermitian metric, and 
\[
\sup _X |T|^2e^{-\lambda} \le 1.
\]
Let $H \to X$ be a holomorphic line bundle with singular Hermitian metric $e^{-\vp}$ such that $e^{-\vp}|_Z$ is still a singular Hermitian metric.  Assume that 
\[ 
\ii (\di \dbar \vp  +{\rm Ricci}(\omega)) \ge \ii \di \dbar \lambda _Z
\]
and
\[
\ii (\di \dbar \vp +{\rm Ricci}(\omega)) \ge(1+ \delta) \ii \di \dbar \lambda _Z
\]
for some positive constant $\delta \le 1$.  Then for any section $f \in H^0(Z,H)$ satisfying 
\[
\int _Z \frac{|f|^2e^{-\vp}}{|dT|_{\omega}^2e^{-\lambda }}dA_{\omega} <+\infty 
\]
there exists a section $F\in H^0(X,H)$ such that 
\[
F|_Z=f \quad \text{and} \quad \int _X |F|^2e^{-\vp} dV_{\omega} \le \frac{24\pi}{\delta}\int _Z \frac{|f|^2e^{-\vp}}{|dT|_{\omega}^2e^{-\lambda }}dA_{\omega}.
\]
\end{d-thm}

\subsection{Weights with bounded Laplacian}\label{holo-recentering}

We shall need some weighted $L^2$ estimates in the setting where the weights have bounded Laplacian.  With the exception of Lemma \ref{unif-1pt-interp-eucl}, we shall omit the proofs and settle for references.

\begin{lem}\label{ddbar-est-eucl}
For each $r>0 $ there exists a constant $C=C_{r}>0$ with the following property.  For any $\sC^2$-smooth $(1,1)$-form $\theta$ satisfying 
\[
- M \omega _o \le \theta \le M \omega _o,
\]
and any $z \in \C$ there exists $u \in \sC ^2 (D^o_{2r}(z))$ such that 
\[
\Delta u = \theta \quad \text{and} \quad \sup _{D^o_r(z)} (|u|+|du|_{\omega _P}) \le C M.
\]
\end{lem}

As a corollary, one has the following lemma.

\begin{lem}\label{weight-centering-eucl}
Let $\vp \in \sC ^2(\C)$ satisfy 
\[
-M\omega _o \le \Delta \vp \le M \omega _o
\]
for some positive constant $M$.  Then for any $r >0$ there exists a constant $C=C_r$ such that for any $z \in \C$ there is a holomorphic function $F \in \co (D^o_r(z))$ satisfying 
\[
F(z) = 0, \quad \quad |2\re F (\zeta) - \vp(\zeta) + \vp (z)| \le C, \quad \text{and} \quad |2 \re dF(\zeta) - d\vp(\zeta)| \le C 
\]
for all $\zeta \in D^o_r(z)$.  The constant $C$ depends only on $r$ and $M$, and not on $\vp$ or $z$.
\end{lem}

\noi For the proofs of Lemmas \ref{ddbar-est-eucl} and \ref{weight-centering-eucl}, see, for example, \cite{sv2}.

Lemma \ref{weight-centering-eucl} gives the following  generalizations of Bergman's inequality.

\begin{prop}\label{bergman-eucl}
Let $\vp \in \sC ^2(\C)$ satisfy 
\[
-M \omega _o \le  \Delta \vp \le M\omega _o.
\]
Then for each $r > 0$ there exists $C_r=C_r(M)$ such that for all $f \in \sh ^2(\C, e^{-\vp} \omega _o)$, 
\begin{enumerate}
\item[(a)] 
\[
|f(z)|^2e^{-\vp(z)} \le C_r \int _{D^o_r(z)} |f|^2e^{-\vp} \omega _o,
\]
and 
\item[(b)]
\[
|d(|f|^2e^{-\vp})(z)| \le C_r \int _{D^o_r(z)} |f|^2e^{-\vp} \omega _o.
\]
\end{enumerate}
\end{prop}

\noi For the proof, see \cite{quimseep}.

\begin{cor}\label{Bergman-sums-eucl}
If $\Gamma$ is a finite union of uniformly separated sequences then 
\begin{enumerate}
\item[(a)] 
\[
\sum _{\gamma\in \Gamma} |f(\gamma)|^2e^{-\vp(\gamma)} \le C_r \sum _{\gamma\in \Gamma} \int _{D^o_r(\gamma)} |f|^2e^{-\vp} \omega _o \le \tilde C_r \int _{\C} |f|^2e^{-\vp} \omega _o,
\]
and 
\item[(b)]
\[
\sum _{\gamma\in \Gamma} |d(|f|^2e^{-\vp})(\gamma)| \le C_r \sum _{\gamma\in \Gamma} \int _{D^o_r(\gamma)} |f|^2e^{-\vp} \omega _o \le \tilde C_r \int _{\C} |f|^2e^{-\vp} \omega _o.
\]
\end{enumerate}
\end{cor}
Finally, we will use the following result.

\begin{lem}\label{unif-1pt-interp-eucl}
Let $\vp \in \sC ^2(\C)$ be a weight function satisfying 
\[
\Delta \vp \ge c\omega _o
\]
for some positive constant $c$.  Then there exists a universal constant $C>0$ such that for any $z\in \C$ there is a function $f \in \sh ^2(\C, e^{-\vp}\omega _o)$ satisfying 
\[
|f(z)|^2e^{-\vp(z)} = 1 \quad \text{and} \quad \int _{\C} |f|^2e^{-\vp} \omega _o \le C/c.
\]
\end{lem}

\begin{proof}
Though proofs can be found in many places, we shall give a new one based on the $L^2$ extension theorem.  To this end, consider the holomorphic function $T_z(\zeta) = \zeta -z$ and the function $\lambda_z :\C \to \R$ defined by 
\[
\lambda _z(\zeta) := \frac{1}{\pi r^2} \int _{D^o_r(\zeta)} \log |x -z|^2 \omega _o(x),
\]
seen respectively as a holomorphic section and a singular Hermitian metric for the line bundle on $\C$ associated to the one-point divisor $z$.  Observe that since $\Delta \vp \ge c \omega _o$, for any $\delta > 0$, we can find $r >>0$ such that 
\[
\Delta \vp + {\rm R}(\omega _o) - (1+\delta) \Delta \lambda_z = \Delta \vp - (1+\delta) \Delta \lambda_z \ge (c - 2(1+\delta)r^{-2}) \omega _o \ge 0.
\]
We can therefore apply Theorem \ref{ot-basic} to obtain an extension of the `function' $f : \{z\} \to \R$ defined by 
\[
f(z) := e^{\vp(z)/2}
\]
to a function $F \in \co (\C)$ satisfying 
\[
\int _{\C} |F|^2e^{-\vp} \omega _o \le \frac{C}{|dT_z(z)|^2_{\omega _o}e^{-\lambda _z(z)}} ,
\]
with $C$ independent of $z$.  Now, $|dT_z(z)|^2_{\omega _o} = 1$, and 
\[
\lambda _z(z) = \frac{1}{\pi r^2} \int _{D^o_r(0)} \log |x|^2 \omega _o(x) = \frac{1}{r^2} \int _0 ^{r^2} \log(t) dt = \log r^2 - 1,
\]
This completes the proof.
\end{proof}

\subsection{Jensen Formula}  

We shall make fundamental use of the following weighted analog of the well-known Jensen Formula, which gives a weighted count of the number of zeros of holomorphic functions in disks.  While the weighted version follows rather easily from the unweighted version, we will give a direct and short proof for the reader's convenience.

\begin{d-thm}[Jensen Formula]\label{Jensen-eucl}
Let $f \in \co (\C)$, let $z \in \C$, and let $r > 0$.  Let $a_1,...,a_N$ denote the zeros of $f$ in $D^o_r(z)$, and assume that $f(z) \neq 0$, and that there are no zeros of $f$ on the boundary of the disk $D^o_r (z)$. Then 
\[
 \frac{1}{2\pi} \int _{\di D^o_r(z)} \!\!\!\! \log (|f|^2e^{-\vp}) d\theta_z =  \log \left (|f(z)|^2e^{-\vp(z)}\right )  +  \sum _{j=1} ^N \log \frac{r^2}{|z-a_j|^2} - \frac{1}{2\pi} \int _{D^o_r(z)} \!\!\!\!\log\left ( \frac{r^2}{|\zeta - z|^2}\right ) \Delta \vp(\zeta)
\]
where $\frac{1}{2\pi} d\theta_z$ is the uniformly distributed probability measure on $\di D^o_r(z)$.
\end{d-thm}

\begin{proof}
Recall that $d^c = \frac{\ii}{2}(\dbar - \di)$, so that $dd^c = \Delta$.  Let 
\[
G_z(\zeta) = \log \frac{|\zeta -z|}{r} \quad \text{and} \quad H(\zeta) = \log \left ( \frac{|f(\zeta)|^2e^{-\vp(\zeta)}}{\prod _{j=1} ^N \frac{|\zeta -a_i|^2}{r^2}} \right ).
\]
Note that $d^c G_z = \tfrac{1}{2} d\theta_z$.  By Stokes' Theorem we have 
\begin{equation}\label{green-id-here}
\int _{\di D^o_r(z)} H d^c G_z - G_zd^c H = \int _{D^o_r(z)} H\Delta G_z - G_z \Delta H.
\end{equation}
Now, $G_z|_{\di D^o_r(z)}\equiv 0$ and $\Delta G_z = \pi \delta _z$.  It follows that 
\begin{eqnarray*}
\frac{1}{\pi } \int _{\di D^o_r(z)} \log (|f|^2e^{-\vp}) d^c G_z &=&  \log \left (|f(z)|^2e^{-\vp(z)}\right ) +  \sum _{j=1} ^N   \left (  \log \frac{r^2}{|z-a_i|^2} + 2 \int _{\di D^o_r(z)} G_{a_j} d^c G_z \right )\\
&&  - \frac{1}{\pi}\int _{D^o_r(z)} \log \frac{r}{|\zeta - z|} \Delta \vp(\zeta).
\end{eqnarray*}
But since $G_z|_{\di D^o_r(z)} \equiv 0$, and application of \eqref{green-id-here} with $H= G_{a_j}$ gives
\[
\int _{\di D^o_r(z)} G_{a_j} d^c G_z =  \int _{D^o_r(z)} G_{a_j} \Delta G_z - G_z \Delta G_{a_j} = G_{a_j}(z) - G_z(a_j) = 0,
\]
and thus the result follows.
\end{proof}

\section{Shapiro-Shields Interpolation}\label{SSI}

Strictly speaking, this section of the article is not necessary for the proof of the main result, and may be skipped. However, the discussion ties the problem we are studying to a more classical approach to interpolation in Hilbert spaces of holomorphic functions introduced by Shapiro and Shields.

A Hilbert space of holomorphic functions is a Hilbert space $\sh$ consisting of holomorphic functions on some complex manifold $X$, with the additional property that point evaluation is a bounded linear functional.  The boundedness of point evaluation implies, via the Riesz Representation Theorem, that for each $z \in X$ there is an element $K_z\in \sh$ such that 
\[
f (z) = \left < f|K_z\right >.
\]
In particular, 
\[
K_{\zeta}(z) = \left < K_{\zeta}|K_z\right > = \overline{\left < K_z|K_{\zeta}\right >} = \overline{K_z(\zeta)}.
\]
Moreover, if $||f||=1$ then  
\[
|f(z)|^2 \le  ||K_z||^2 = \left < K_z |K_z\right > = K(z,z) ,
\]
with equality if and only if $f$ is is a unimodular multiple of $K_z$.  Consequently we have the extremal characterization
\begin{equation}\label{bk-ext-char}
K(z,z) = \sup \left \{ |f(z)|^2\ ;\ f \in \sh \text{ and }||f||=1\right \}.
\end{equation}
The function $K(\zeta,z) := K_z (\zeta)$ is often called the {\it kernel function}.

In \cite{ss}, Shapiro and Shields considered the general interpolation problem for Hilbert spaces of holomorphic functions, giving the following definition.

\begin{defn}[Shapiro-Shields Interpolation Sequences]
A sequence $\Gamma \subset X$ is said to be interpolating in the sense of Shapiro-Shields if for each $c = (c_{\gamma}) \in \ell ^2$ there exists $f \in \sh$ such that 
\[
\frac{f(\gamma)}{\sqrt{K(\gamma , \gamma)}} = c_{\gamma}
\]
for all $\gamma \in \Gamma$.
\red
\end{defn}

One can rephrase the problem of whether a sequence is interpolating in the sense of Shapiro-Shields as follows.  Let $\Gamma \subset X$ be a closed discrete subset.  Consider the space 
\[
\ell ^2_K(\Gamma) := \left \{ a : \Gamma \to \C\ ;\ \sum _{\gamma \in \Gamma} \frac{|a(\gamma)|^2}{K(\gamma , \gamma)} < +\infty \right \}.
\]
Then $\Gamma$ is interpolating in the sense of Shapiro-Shields if and only if the map of restriction to $\Gamma$ 
\[
\mathbf{R} _{\Gamma} : \sh \to \ell ^2_K(\Gamma)
\]
is surjective.

Let us now return to our setting.  Our Hilbert space of holomorphic functions is the generalized Bergman space 
\[
\sh = \sh ^2(X, e^{-\vp} \omega).
\]
In our case, $\sh$ is a closed subspace of the Hilbert space
\[
L^2 (X, e^{-\vp}\omega)
\]
of measurable functions $f :X \to \C$ satisfying
\[
\int _X |f|^2 e^{-\vp} \omega < +\infty.
\]
The fact that $\sh^2(X, e^{-\vp}\omega)$ is closed follows from the boundedness of point evaluation and Montel's Theorem.  Consequently, there is a bounded orthogonal projection 
\[
P : L^2 (X, e^{-\vp}\omega) \to \sh ^2 (X, e^{-\vp}\omega),
\]
called the Bergman projection.  By Schwartz's Kernel Theorem, the Bergman projection is an integral operator given by 
\[
(Pf)(z) = \int _X f(\zeta) K(z,\zeta) e^{-\vp}(\zeta) \omega (\zeta).
\]
In particular, since $P$ is the identity on $\sh^2(X, e^{-\vp}\omega)$, $K$ is the Kernel function for $\sh$.

We then have the following well-known proposition.

\begin{prop}\label{ss=us}
Let $(X, \omega)$ be an asymptotically flat finite Riemann surface.  Suppose the weight function $\vp$ satisfies the curvature hypothesis \eqref{curv-bd-gen}.  Then there exists a positive constant $C$ such that 
\begin{equation}\label{bk-est}
C^{-1} \le K(z,z) e^{-\vp(z)} \le C.
\end{equation}
In particular, the two Hilbert spaces $\ell ^2_K(\Gamma)$ and $\ell ^2 (\Gamma , e^{-\vp})$ are quasi-isomorphic as Hilbert spaces, and equal as subsets of $\co (X)$.
\end{prop}

\noi Thus we see that our notion of interpolation sequences agrees with the notion introduced by Shapiro and Shields, which lends further justification for our choice of the definition of the Hilbert space $\ell ^2 (\Gamma, e^{-\vp})$.

\medskip

In the proof of Proposition \ref{ss=us} we shall make use of the following global version of Lemma \ref{unif-1pt-interp-eucl} (with slightly stronger curvature hypotheses).

\begin{lem}\label{unif-1pt-interp-gen}
Let $(X, \omega)$ be an asymptotically flat finite Riemann surface, and let $\vp \in \sC ^2(X)$ be a weight function satisfying the curvature hypothesis \eqref{curv-bd-gen}.  Then there exists a constant $C>0$ such that for any $z\in X$ there is a function $f \in \sh ^2(X, e^{-\vp}\omega )$ satisfying 
\[
|f(z)|^2e^{-\vp(z)} = 1 \quad \text{and} \quad \int _{X} |f|^2e^{-\vp} \omega  \le C.
\]
\end{lem}

\begin{proof}
Again, this result is not original; a proof based on Ohsawa-Takegoshi can be given in this case as well (and then the upper curvature bound in \eqref{curv-bd-gen} is not needed), but for the sake of variety we shall give the more traditional proof, introduced in \cite{quimbo}, using H\"ormander's Theorem and the holomorphic recentering of weights discussed in Paragraph \ref{holo-recentering}.  

Fix $z \in X$.  Because $(X,\omega)$ is asymptotically flat, we can find an open subset $U$ containing $z$ such that $\int _U \omega$ is bounded independent of $z$, and a biholomorphism $g : U \to \D$ to the unit disk, such that 
\[
g(z) =0 \quad \text{and} \quad \frac{1}{C_o} \le |dg|_{\omega} \le C_o \text{ on }U,
\]
where $C_o$ is a constant that is independent of $z$.  Moreover, since $\ii \di \dbar (g^* \vp) = g^* \ii \di \dbar \vp$, Lemma \ref{weight-centering-eucl} provides a function $F \in \co (U)$ satisfying 
\[
F(z) = 0 \quad \text{and} \quad  \left |2 \re F - \vp + \vp (z)\right | \le C_1
\]
for some constant $C_1$ independent of $z$.  (It is here that we need the upper curvature bounds in \eqref{curv-bd-gen}.)

Let $\chi \in \sC ^{\infty} _0(\D)$ with $\chi \equiv 1$ on $\frac{1}{2} \D$ and $0 \le \chi \le 1$.  Then the function 
\[
\tilde f  = e^{\vp(z)/2} e^F \chi \circ g^{-1} 
\]
has support in $U$ and satisfies $|\tilde f(z)|^2e^{-\vp(z)}=1$, and we have the estimate 
\[
\int _X |\tilde f|^2e^{-\vp} \le  \int _U e^{\vp (z) + 2\re F - \vp} \omega \le c_o
\]
where $c_o$ is independent of $z$.  

Next, since $\chi ' \equiv 0$ on $\frac{1}{2} \D$, we estimate that
\[
\int _{X} \frac{|\dbar \tilde f|^2}{|g|^2}e^{-\vp}\omega \lesssim \int _U e^{\vp (z) +2 \re F - \vp} |dg|^{-2}_{\omega} |\chi '\circ g^{-1}|^2 \omega \le c_1,
\]
where $c_1$ is again independent of $z$.  Since 
\[
\ii \di \dbar \vp + {\rm R}(\omega) \ge m \omega,
\]
H\"ormander's Theorem provides a function $u: X \to \C$ such that $\dbar u = \dbar \tilde f$ and 
\[
\int _X |u|^2 e^{-\vp} \omega \le \int _X \frac{|u|^2e^{-\vp}}{|g|^2} \omega \le c_2,
\]
with $c_2$ again independent of $z$.  In particular, in conjunction with the elliptic regularity of $\dbar$, the second inequality implies that $u(z) = 0$.  Finally, we set 
\[
f = \tilde f - u.
\]
Then 
\[
|f(z)|^2 e^{-\vp(z)} = |\tilde f (z)|^2 e^{-\vp(z)} = 1,
\]
and 
\[
\int _X |f|^2 e^{-\vp} \omega \le 2\left ( \int _X |\tilde f|^2 e^{-\vp} \omega + \int _X |u|^2 e^{-\vp} \omega\right )  \le C
\]
where $C$ is independent of $z$.  The proof is complete.
\end{proof}

\begin{proof}[Proof of Proposition \ref{ss=us}]
Since $(X,\omega)$ is asymptotically flat and the inequality holds trivially on any relatively compact domain $D \relcomp X$ (with the constant depending on $D$) it suffices to show the estimate \eqref{bk-est} at the ends.  Then we are in two possible cases:  (i) The point $z$ lies in a Euclidean end, and (ii) the point $z$ lies in a cylindrical end.  In fact, case (ii) can be reduced to case $(i)$ because the universal convering map is an isometry between the cylindrical and Euclidean metrics, and the pullback of $\vp$ by the universal covering map satisfying the curvature hypothesis \eqref{curv-bd-gen}.  Note, moreover, that since $\omega$ is zero outside a compact set, the hypothesis \eqref{curv-bd-gen} becomes $m \omega _o \le \vp \le M\omega _o$.  Consequently Proposition \ref{bergman-eucl}(a) and the extremal characterization \eqref{bk-ext-char} of the Bergman kernel imply that $K(z,\bar z) e^{-\vp(z)} \le C_1$.  Finally, the inequality $K(z,\bar z) e^{-\vp(z)} \ge C_2$ is an immediate consequence of Lemma \ref{unif-1pt-interp-gen}.  
\end{proof}

\section{Interpolation in $(\C, \omega_o)$}\label{Eucl-section}

\subsection{The interpolation theorem}
Recall that 
\[
\A ^o_r(z) := \{ \zeta \in \C\ ;\ 1 < |\zeta - z| < r \}.
\]
In this section we establish the following result. 

\begin{d-thm}\label{char-eucl}
Let $\vp \in \sC ^2 (\C)$ be a weight function satisfying 
\[
0 \le \Delta \vp \le M\omega _o \quad \text{and} \quad \frac{1}{\pi r^2} \int _{D^o_r(z)} \log \frac{r^2}{|\zeta -z|^2} \Delta \vp (\zeta) \ge m 
\]
for some positive constants $M$ and $m$, and let $\Gamma \subset \C$ be a closed discrete subset.  Then the restriction map $\sr_{\Gamma} :\sh ^2 (\C , e^{-\vp}\omega _o ) \to \ell ^2 (\Gamma, e^{-\vp})$ is surjective if and only if 
\begin{enumerate}
\item[(i)] $\Gamma$ is uniformly  separated with respect to the Euclidean distance, and 
\item[(ii)] the upper density 
\[
D^+_{\vp} (\Gamma) := \limsup _{r \to \infty} \sup _{z \in \C} \frac{2\pi \int _{\A^o_r(z)} \log \frac{r^2}{|\zeta-z|^2} \delta _{\Gamma}}{\int _{D^o_r(z)} \log \frac{r^2}{|\zeta-z|^2}\Delta \vp(\zeta)}
\]
satisfies $D^+_{\vp}(\Gamma) < 1$.  
\end{enumerate}
\end{d-thm}

\begin{s-rmk}
The sufficiency of conditions (i) and (ii) follows from work of the author and V. Pingali, which we will recall below, giving a slightly simpler proof in the present setting. The necessity of conditions (i) and (ii) were essentially established by Ortega Cerd\`a and Seip \cite{quimseep}, and we will adapt their methods here.
\red
\end{s-rmk}

\begin{s-rmk} 
In line with the remark following Definition \ref{Eucl-density-defn}, Theorem \ref{char-eucl} with $D^+_{\vp}(\Gamma)$ replaced by the non-logarithmic version of the density, as given, for example, in Theorem \ref{quimbo-thm} above.  Essentially, for sufficiency one simply removes the logarithms.   For necessity, one has to establish the above result, and use the argument alluded to by Ortega Cerd\`a and Seip to show that the two densities are the same in the Euclidean case.
\red
\end{s-rmk}

It is useful to define the {\it Euclidean separation radius} 
\[
R^o_{\Gamma} := \inf\left  \{ \frac{|\gamma_1-\gamma _2|}{2}\ ;\ \gamma _1 , \gamma _2 \in \Gamma,\ \gamma _1 \neq \gamma _2\right \}
\]
of $\Gamma$.  Of course, the Euclidean separation radius of $\Gamma$ is positive if and only if $\Gamma$ is uniformly separated in the Euclidean distance.

\subsection{Weights and density}

We will use the fact that 
\begin{equation}\label{avg-of-log}
\int _{D^o_r(0)} \log \frac{r^2}{|\zeta|^2} \omega _o(\zeta) = \pi r^2.
\end{equation}

\begin{prop}\label{avged-wt}
Let $\vp \in \sC ^2 (\C)$ be a weight function satisfying 
\[
- M\omega _o \le \Delta \vp \le M \omega _o,
\]
and let 
\[
\vp _r (z) :=  \frac{1}{\pi r^2}\int _{D^o_r(z)} \vp (\zeta) \log \frac{r^2}{|\zeta -z|^2} \omega _o (\zeta)= \frac{1}{\pi r^2} \int _{D^o_r(0)}  \vp (\zeta+z) \log \frac{r^2}{|\zeta|^2}\omega _o(\zeta) , \qquad z \in\C.
\]
Then 
\[
-  M \omega _o \le \Delta \vp_r \le M  \omega _o,
\]
and there is a constant $C_r>0$ such that for all $z \in \C$, 
\[
|\vp (z) - \vp _r(z)| \le C_r.
\]
In particular, we have the following quasi-isometries 
\[
\sh ^2 (\C , e^{-\vp}\omega _o) \asymp \sh ^2 (\C , e^{-\vp_r} \omega _o) \quad \text{and} \quad \ell ^2(\Gamma, e^{-\vp}) \asymp \ell ^2 (\Gamma, e^{-\vp_r}).
\]
of Hilbert spaces given by the identity map.
\end{prop}

\begin{proof}
The estimates for $\Delta \vp _r$ are clear from \eqref{avg-of-log} and the second integral formula for $\vp _r$.  Next, by Proposition \ref{ddbar-est-eucl} there is a function $u \in \sC ^{2}(D^o_{2r}(z))$ such that 
\[
\Delta u = \Delta \vp  \quad \text{and} \quad \sup _{D^o_r(z)} |u| \le \frac{C_r}{2},
\]
with $C_r$ independent of $z$.  Let 
\[
h_z (\zeta) := \vp  (\zeta) - u(\zeta) - (\vp (z) - u(z)).
\]
Then $h_z$ is harmonic in $D^o_2(z)$ and vanishes at $z$.  It follows that 
\begin{eqnarray*}
\left | \vp (z) - \frac{1}{\pi r^2} \int _{D^0_r(z)}\vp (\zeta) \log \frac{r^2}{|\zeta -z|^2} \omega _o (\zeta)\right | &=&\left |\frac{1}{\pi r^2} \int _{D^o_r(z)} (h_z(\zeta) + u(\zeta) - u(z)) \log \frac{r^2}{|\zeta -z|^2}\omega _o(\zeta) \right | \\
&=&\left |\frac{1}{\pi r^2} \int _{D^o_r(z)}  (u(\zeta) - u(z)) \log \frac{r^2}{|\zeta -z|^2} \omega _o(\zeta) \right | \\
&\le & C_r,
\end{eqnarray*}
as claimed.
\end{proof}

Recall that 
\[
c_r := \int _{\A^o_r(0)} \log \frac{r^2}{|\zeta|^2}\omega _o(\zeta)
\]
and that  
\begin{equation}\label{lambda-eucl}
\lambda _r ^T (z) :=  \frac{1}{c_r}\int _{\A^o_r(z)} \log |T(\zeta)|^2 \log \frac{r^2}{|\zeta -z|^2} \omega _o (\zeta) = \frac{1}{c_r} \int _{\A^o_r(0)} \log |T(\zeta +z)|^2 \log \frac{r^2}{|\zeta|^2}\omega _o (\zeta).
\end{equation}

\begin{prop}\label{geometric-invariants}
Fix $T \in \co (\C)$ such that ${\rm Ord}(T) = \Gamma$.
\begin{enumerate}
\item[(a)] The functions
\[
\sigma ^{\Gamma}_r := |T|^2 e^{-\lambda ^T_r} : \C \to (0,\infty) \quad \text{and} \quad S^{\Gamma}_r := |dT|^2_{\omega _o} e^{-\lambda ^T_r} : \Gamma \to (0,\infty),
\]
and the $(1,1)$-form 
\[
\Upsilon ^{\Gamma}_r := \frac{1}{2\pi} \Delta \lambda ^T_r, 
\]
are independent of the choice of $T$.  In fact, $\sigma ^{\Gamma}_r (z)$ and $S^{\Gamma}_r(\gamma)$ depend only on the finite sets $\Gamma \cap D^o_r(z)$ and $\Gamma \cap D^o_r(\gamma)$ respectively.
\item[(b)]  The inequality $\sigma ^{\Gamma}_r \le 1$ holds.  Moreover, $\sigma ^{\Gamma}_r (z) = 1$ as soon as $D^o_r(z) \cap \Gamma = \emptyset$.
\item[(c)] For any $\gamma \in \Gamma$ and any $z \in D^o_{R^o _{\Gamma}}(\gamma)$ such that $|z-\gamma| > \ve$, we have the estimate 
\[
\sigma ^{\Gamma} _r (z) \ge C_r \ve ^2.
\]
On the other hand, $\frac{1}{\sigma ^{\Gamma}_r}$ is not locally integrable in any neighborhood of any point of $\Gamma$.
\item[(d)] One has the formula 
\[
\Upsilon ^{\Gamma}_r (z) =  \frac{1}{c_r} \left ( \int_{\A^o_r(z)} \log \frac{r^2}{|\zeta-z|^2}\delta _{\Gamma}(\zeta)\right ) \omega _o(z).
\]
\end{enumerate}
\end{prop}

\begin{proof}
If $\tilde T$ is another holomorphic function with ${\rm Ord}(\tilde T)=\Gamma$ then $\tilde T = e^h T$ for some $h \in \co (\C)$, and thus 
\[
\lambda ^{\tilde T}_r = 2\re h + \lambda ^T _r.
\]
Thus (a) follows.  The sub-mean value property implies that $\sigma^{\Gamma}_r \le 1$, and if $D^o_r(z) \cap \Gamma=\emptyset$ then $\log |T|^2 |_{D^o_r(z)}$ is harmonic, and thus by the mean value property for harmonic functions we have $\sigma ^{\Gamma}_r(z) =1$.  Thus (b) holds.   To prove (c), fix $\gamma \in \Gamma$ and $z \in \C$ with $\ve \le |z-\gamma| \le R^o _{\Gamma}$.  By (a), we may use the function 
\[
T(\zeta) = \prod _{\mu \in \Gamma \cap D^o_r(z)-\{ \gamma\}} z-\mu
\]
to cut out $\Gamma \cap D^o_r(z)$.  For ease of notation, let us write $\Gamma \cap D^o_r(z) - \{\gamma\} = \{\mu _1,...,\mu _{N_r}\}$.  We note that since $\Gamma$ is uniformly separated, $N_r$ is uniformly bounded independent of $z$.  

With this choice of $T$, we have 
\begin{eqnarray*}
\log \sigma _r ^{\Gamma}(z)  &=& \log |z-\gamma|^2 - \frac{1}{c_r}\int _{\A^o_r(z)} \log |\zeta  - \gamma|^2 \log \frac{r^2}{|\zeta -z|^2} \omega _o (\zeta) \\
&& + \sum _{j=1} ^{N_r} \log |z -\mu_j |^2 - \frac{1}{c_r}\int _{\A^o_r(z)} \log |\zeta  - \mu_j|^2 \log \frac{r^2}{|\zeta -z|^2}\omega _o (\zeta).
\end{eqnarray*}
Now, $\log |z-\mu _j|^2 \ge \log (R^o _{\Gamma})^2$ for $1 \le j \le N_r$, while 
\[
\frac{1}{c_r} \int_{\A^o_r(z)} \log |\zeta - x|^2 \log \frac{r^2}{|\zeta -z|^2} \omega _o (\zeta) \le \log r^2 \quad \text{for }x=\gamma, \mu _1, ..., \mu _{N_r}.
\]
It follows that 
\[
\sigma ^{\Gamma} _r (z) \ge  r^{2(N_r+1)} (R^o _{\Gamma})^{2N_r} |z-\gamma|^2,
\]
and therefore we have (c).

Finally, (d) is a consequence of the Lelong-Poincar\'e formula 
\[
\frac{1}{2\pi} \Delta \log |T|^2 = \delta _{\Gamma}
\]
in the sense of distributions.
\end{proof}

\subsection{Sufficiency}

In this section we present the following result, which is only a slight modification of a theorem from \cite{pv}.  

\begin{d-thm}[Strong sufficiency: Euclidean case]\label{s-suff-eucl}
Let $\vp :\C \to [-\infty, \infty)$ be any subharmonic weight.  Assume that $\Gamma \subset \C$ is uniformly separated with respect to the Euclidean distance, and that
\begin{equation}\label{density-ot-style}
\Delta \vp  \ge 2\pi \alpha \Upsilon^{\Gamma}_r 
\end{equation}
for some $r >0$ and $\alpha >1$.  Then the restriction $\sr_{\Gamma} :\sh ^2 (\C , e^{-\vp}\omega _o ) \to \ell ^2 (\Gamma, e^{-\vp})$ is surjective.
\end{d-thm}

First, we apply Theorem \ref{ot-basic} to the case at hand.  Let $(X,\omega) = (\C, \omega _o)$, choose $T \in \co (\C)$  with ${\rm Ord}(T)=\Gamma$, and take $\lambda := \lambda ^T_r$ as in \eqref{lambda-eucl}.  Then $|T|^2e^{-\lambda} \le 1$, and thus the curvature conditions of Theorem \ref{ot-basic} mean exactly that $D^+_{\vp}(\Gamma) < 1$ implies the following result.

\begin{d-thm}
Let $\vp$ be a plurisubharmonic function on $\C$, and let $\Gamma \subset \C$ be any closed discrete subset.  Assume that 
\[
\Delta \vp  \ge 2\pi \alpha \Upsilon^{\Gamma}_r 
\]
for some $\alpha >1$.  Then for any $f : \Gamma \to \C$ satisfying 
\[
\sum _{\gamma \in \Gamma} \frac{|f(\gamma)|^2e^{-\vp(\gamma)}}{S_r^{\Gamma}(\gamma)} < +\infty
\]
there exists $F \in \sh ^2 (\C, e^{-\vp}\omega _o)$ such that 
\[
F|_{\Gamma} = f \quad \text{and} \quad \int _{\C} |F|^2e^{-\vp} \omega _o \le \frac{24\pi}{\alpha} \sum _{\gamma \in \Gamma} \frac{|f(\gamma)|^2e^{-\vp(\gamma)}}{S_r^{\Gamma}(\gamma)}.
\]
\end{d-thm}

To finish the proof of Theorem \ref{s-suff-eucl}, it suffices to prove the following result.

\begin{prop}\label{denom-bds-eucl}
Let $\Gamma \subset \C$ be a closed discrete subset.  Then $\Gamma$ is uniformly separated with respect to the Euclidean distance if and only if for any $r > 1$ there exists $C_r>0$ such that 
\[
\inf _{\gamma \in \Gamma} S_r^{\Gamma}(\gamma) \ge C_r.
\]
\end{prop}

\begin{proof}
Since, by Proposition \ref{geometric-invariants}(a), for each $\gamma _o \in \Gamma$ the quantity $S_r^{\Gamma}(\gamma_o)$ depends only on finite subset 
\[
\Gamma _r (\gamma _o) := \{ \gamma \in \Gamma\ ;\ |\gamma _o-\gamma| < r \}
\]
of $\Gamma$,  we may use any holomorphic function $T$ that vanishes on $\Gamma _r(\gamma _o)$.  Let us fix $\gamma _o$, then, and enumerate the points of $\Gamma _r(\gamma _o)$ as $\gamma _o, \gamma _1,...,\gamma _N$, in such a way that $\gamma _1$ is the (not necessarily unique) closest point of $\Gamma - \{\gamma _o\}$ to $\gamma _o$.  We take the function 
\[
T(z) = \prod _{j=o} ^N z-\gamma _j.
\]
Note that 
\[
|dT(\gamma_o)|_{\omega _o}^2 = \prod _{j=1} ^N |\gamma_j- \gamma _o|^2.
\]

Now suppose $\Gamma$ is uniformly separated in the Euclidean distance.  Then the number $N= N(\gamma _o)$ is uniformly bounded for each $r$,  independent of $\gamma _o$, and we have 
\[
|dT(\gamma_o)|_{\omega _o}^2 \ge (R^o _{\Gamma})^N.
\] 
On the other hand, since $|T|< r^N$, 
\[
\lambda _r ^T < N \log r.
\]
Thus we see that 
\[
|dT(\gamma _o)|^2_{\omega _o} e^{-\lambda _r ^T(\gamma _o)} \ge C_r
\]
where $C_r$ depends only on $r$.

In the other direction,  observe first that since $|\gamma _j - \zeta| < 2r$ for all $\zeta \in \A^o_r(\gamma _o)$,
\begin{eqnarray*}
\int _{\A^o _r(\gamma _o)}\left ( \log \frac{|\zeta -\gamma _j |^2}{4r^2} \right )\log \frac{r^2}{|\zeta -\gamma _o|^2} \omega _o (\zeta) & \ge & \log 4 \int _{\A^o _r(\gamma _o)}\left ( \log \frac{|\zeta -\gamma _j |^2}{4r^2} \right )\omega _o (\zeta)\\
& \ge & \log 4 \int _{D^o _{2r}(\gamma _j)}\left ( \log \frac{|\zeta -\gamma _j |^2}{4r^2} \right )\omega _o (\zeta)\\
& = & \log 4 \int _{D^o _{2r}(0)}\left ( \log \frac{|\zeta|^2}{4r^2} \right )\omega _o (\zeta),
\end{eqnarray*}
and therefore 
\[
\frac{1}{c_r} \int _{\A^o _r(\gamma _o)}\left ( \log \frac{|\zeta -\gamma _j |^2}{4r^2} \right )\log \frac{r^2}{|\zeta -\gamma _o|^2} \omega _o (\zeta) \ge -A_r
\]
for some constant $A_r$ that is independent of $\Gamma$.  We therefore have 
\[
\lambda _r ^T(\gamma _o) = \sum _{j=o} ^N \frac{1}{c_r} \int _{\A^o _r(\gamma _o)}\left ( \log |\zeta -\gamma _j |^2 \right )\log \frac{r^2}{|\zeta -\gamma _o|^2} \omega _o (\zeta)\\
\ge -(N+1)A_r \ge - M_r,
\]
Where again $M_r$ depends only on $r$.  We therefore have 
\[
C_r e^{M_r} \le |dT(\gamma_o)|^2_{\omega _o} e^{M_r-\lambda _r ^T (\gamma _o)} \le |dT(\gamma_o)|_{\omega _o}^2 = \prod _{j=1} ^N |\gamma _o-\gamma _j|^2 \le r^{2(N-1)} |\gamma _o-\gamma_1|^2.
\]
Thus 
\[
|\gamma _o - \gamma _1| \ge r^{1-N}\sqrt{C_re^{M_r}},
\]
and the proof is thus finished.
\end{proof}

\begin{rmk}\label{uniformity-of-denom-bd}
Notice that if $\Gamma$ is uniformly separated then the constant $C_r$ depends only on the separation radius $R^o_{\Gamma}$, and not on $\Gamma$ itself.  That is to say, the same constant $C_r$ works for all sequences whose separation constant is $\ge R^o_{\Gamma}$.
\red
\end{rmk}

Finally, we observe that if, in place of $\vp$, we use the function $\vp _r$ defined by \eqref{log-avg-defn}, then Theorem \ref{s-suff-eucl} implies the `if' direction of Theorem \ref{char-eucl}. Indeed,  if we replace $\vp$ by $\vp _r$ in Theorem \ref{s-suff-eucl}, then Condition \eqref{density-ot-style} is equivalent to the condition $D^+_{\vp}(\Gamma) < 1$. 
\qed

\subsection{Necessity}

To complete the proof of Theorem \ref{char-eucl}, we  establish the following result, whose proof occupies the final part of this section.

\begin{d-thm}[Necessity: Euclidean case]\label{necess-eucl}
Let $\vp \in \sC ^2(\C)$ be a weight function satisfying 
\[
0 \le \Delta \vp \le M\omega _o \quad \text{and} \quad \frac{1}{\pi r^2} \int _{D^o_r(z)} \log \frac{r^2}{|\zeta -z|^2} \Delta \vp (\zeta) \ge m 
\]
for some positive constants $m$ and $M$, and let $\Gamma \subset \C$ be a closed discrete subset.  If 
\[
\sr _{\Gamma}:\sh ^2(\C , e^{-\vp}\omega _o) \to \ell ^2 (\Gamma ,e^{-\vp})
\]
is surjective, then $\Gamma$ is uniformly separated and $D^+_{\vp}(\Gamma) < 1$.
\end{d-thm}

For the rest of this section, we assume that our weight $\vp$ is as in Theorem \ref{necess-eucl}.

\subsubsection{\sf The interpolation constant}
Observe that if the restriction map $\sr _{\Gamma} : \sh ^2(\C ,e^{-\vp} \omega _o) \to \ell ^2 (\Gamma, e^{-\vp})$ is surjective, then it has a bounded section (i.e., there exists a bounded extension operator).  The argument is as follows.  First, for each $f \in \ell ^2 (\Gamma, e^{-\vp})$ take the extension $\se _{\Gamma}(f)$ that is orthogonal to the kernel of $\sr _{\Gamma}$.  Since ${\rm Kernel}(\sr _{\Gamma})$ is a closed subspace of $\sh ^2 (\C , e^{-\vp}\omega _o)$, this extension is well-defined, and is in fact the (unique) extension of minimal norm in $\sh^2(\C, e^{-\vp}\omega_o)$.  Note that $\se _{\Gamma}: \ell ^2(\Gamma, e^{-\vp}) \to {\rm Kernel}(\sr _{\Gamma})^{\perp}$ has closed graph: indeed, if $f_j \to f$ in $\ell ^2(\Gamma, e^{-\vp})$ and $\se _{\Gamma}(f_j) \to G$ in $\sh ^2(\C, e^{-\vp}\omega_o)$, then since convergence in $\sh ^2(\C, e^{-\vp}\omega _o)$ implies locally uniform convergence, $G$ extends $f$.  Furthermore, since $G\in {\rm Kernel}(\sr _{\Gamma})^{\perp}$, we must have $G = \se_{\Gamma}(f)$ by the uniqueness of the minimal extension.  Boundedness of $\se_{\Gamma}$ now follows from the Closed Graph Theorem.

\begin{defn}
Let $\Gamma$ be an interpolation sequence.  The number
\[
\sa _{\Gamma} = \inf \{ A\ ;\ \exists \ E : \ell ^2 (\C , e^{-\vp}) \to \sh ^2 (\C , e^{-\vp}\omega _o) \text{ with }\sr _{\Gamma}E ={\rm Id} \text{ and }  ||E|| \le A\}
\]
is called the {\it interpolation constant} of $\Gamma$.
\red
\end{defn}
Note that in fact, $\sa _{\Gamma} = ||\se _{\Gamma}||$.

\subsubsection{\sf Necessity of uniform separation}
Suppose $\Gamma$ is an interpolation sequence, and let $\gamma _1, \gamma_2 \in \Gamma$ be any two distinct points.  Consider the $\ell^2(\Gamma, e^{-\vp})$-datum $f:\Gamma \to \C$ defined by 
\[
f (\gamma_1) = e^{\vp (\gamma_1)/2}, \quad f(\mu) = 0 \text{ for all }\mu \in \Gamma - \{\gamma _1\}.
\]
Note that $||f||^2 _{\ell ^2(\Gamma, e^{-\vp})} = 1$.  Since $\Gamma$ is an interpolation sequence, the function 
\[
F := \se _{\Gamma}(f) \in \sh ^2(\C , e^{-\vp} \omega _o)
\]
satisfies 
\[
|F(\gamma_1)|^2e^{-\vp(\gamma _1)} = 1, \quad  |F(\gamma_2)|^2e^{-\vp(\gamma_2)} = 0  \quad \text{ and } \quad \int _{\C} |F|^2e^{-\vp} \omega _o \le \sa _{\Gamma}^2 .
\]
By Proposition \ref{bergman-eucl}(b), 
\[
\frac{1}{|\gamma _1 - \gamma_2|} = \frac{|F(\gamma_1)|^2e^{-\vp(\gamma _1)} - |F(\gamma_2)|^2e^{-\vp(\gamma_2)}}{|\gamma _1 - \gamma_2|} \le \sup |d(|F|^2e^{-\vp})| \le C_r \sa _{\Gamma}^2.
\]
Thus any interpolation sequence is uniformly separated.

\subsubsection{\sf Perturbation of interpolation sequences}

In order to estimate the density, we are going to need to be able to perturb our sequences $\Gamma$ a little bit.  We shall do so in two ways.  In the first way, we just perturb the points of $\Gamma$ so that each point moves at most a distance smaller than the separation radius, while in the second way, we add a single point to $\Gamma$.  The upshot is that both sequences remain interpolation sequences, though in the second case we must also to perturb the weight.  The precise results are as follows, and the proofs are simple modifications of proofs of analogous results in \cite{quimseep}.

\begin{prop}\label{jiggle-eucl}
Let $\Gamma \subset \C$ be an interpolation sequence with separation radius $R^o_{\Gamma}$, enumerated as $\Gamma = \{ \gamma _1, \gamma_2, ... \}$, and let $\sa _{\Gamma}$ be the interpolation constant of $\Gamma$.  Suppose $\Gamma '$ is another sequence, such that there exists a constant $\delta \in (0, \min (\sa _{\Gamma}^{-1}, R^o_{\Gamma}))$, and an enumeration $\Gamma ' = \{ \gamma '_1, \gamma '_2, ... \}$ such that 
\[
\sup _{i\in \N} |\gamma _i - \gamma '_i | \le \delta ^2.
\]
Then $\Gamma '$ is also an interpolation sequence, and its interpolation constant is at most 
\[
C\frac{\sa _{\Gamma}}{1-\delta \sa _{\Gamma}},
\]
where $C$ is independent of $\Gamma$ (but depends on the upper bound for $\Delta \vp$).
\end{prop}

\begin{proof}
By Corollary \ref{Bergman-sums-eucl}(b), if $F \in \sh ^2(\C, e^{-\vp}\omega_o)$ then 
\begin{equation}\label{comparison-eucl}
\sum _{j=1} ^{\infty} \left | |F(\gamma_j)|^2e^{-\vp(\gamma_j)} - |F(\gamma'_j)|^2e^{-\vp(\gamma'_j)}\right | \lesssim \delta ^2 \int _{\C} |F|^2e^{-\vp} \omega _o.
\end{equation}
Now let $f \in \ell ^2(\Gamma ',e^{-\vp})$ with $\sum _{j} |f(\gamma_j')|^2e^{-\vp(\gamma_j')} \le 1$.  Since $\Gamma$ is an interpolation sequence, there exist functions $\{ G_j \ ;\ j =1,2,...\} \subset \sh ^2 (\C, e^{-\vp}\omega _o)$ such that 
\[
G_j (\gamma _i) = f(\gamma _i') e^{\frac{1}{2}(\vp (\gamma _i) - \vp(\gamma _i'))}\delta _{ij} \quad \text{and} \quad \int _{\C}  | \sum _{j=1} ^{\infty} G_j | ^2 e^{-\vp} \omega _o \le \sa _{\Gamma} ^2.
\]
(Indeed, we simply take each $G_j$ to be the minimal extension of $g_j(\gamma _i) := f(\gamma _i') e^{\frac{1}{2}(\vp (\gamma _i) - \vp(\gamma _i'))}\delta _{ij}$, and use the fact that the minimal extension operator is linear.)  The function $F = \sum G_j$ does not extend $f$, but a modification of it comes close.  Indeed, by \eqref{comparison-eucl} we have the estimate 
\[
\sum _{j=1} ^{\infty} \left | |f(\gamma _j')|^2e^{-\vp(\gamma_j')} - |G_j(\gamma_j')|^2 e^{-\vp(\gamma_j')}\right | \lesssim \delta ^2 \sa _{\Gamma}^2.
\]
Thus for an appropriate choice of unimodular constants $\alpha _j$, the function 
\[
F_1 := \sum_j \alpha _j G_j
\]
then satisfies the estimate 
\[
\int _{\C} |F_1|^2e^{-\vp} \omega _o \le \sa _{\Gamma}^2 \quad \text{and} \quad \sum _{j=1} ^{\infty} \left | f(\gamma _j')  - F_1(\gamma_j') \right |^2 e^{-\vp(\gamma_j')} \lesssim \delta ^2 \sa _{\Gamma}^2.
\]
The function $F_1$ almost achieves the extension of the datum $f$, so we correct the error inductively as follows.  Set $f_1 = f :\Gamma \to \C$, and let 
\[
f_2 := f_1 - F_1|_{\Gamma}.
\]
Assuming $F_j$ has been found with
\[
F_j(\gamma _i) = f_j(\gamma_i') e^{\frac{1}{2} \left ( \vp(\gamma _i) - \vp (\gamma _i')\right )}, \quad i=1,2,...,
\]
\[
\int _{\C} |F_j|^2e^{-\vp} \omega _o \le \delta ^{2(j-1)} \sa^{2j} _{\Gamma}, \quad \text{and} \quad \sum _{j=1} ^{\infty} \left | f_j(\gamma _j')  - F_j(\gamma_j') \right |^2 e^{-\vp(\gamma_j')} \lesssim (\delta ^2 \sa _{\Gamma}^2)^j,
\]
set $f_{j+1} := f_j - F_j|_{\Gamma}$ and apply the above procedure to obtain $F_{j+1}$ satisfying 
\[
\int _{\C} |F_j|^2e^{-\vp} \omega _o \le \delta ^{2j} \sa^{2(j+1)} _{\Gamma} \quad \text{and} \quad  \sum _{j=1} ^{\infty} \left | f_{j+1}(\gamma _j')  - F_{j+1}(\gamma_j') \right |^2 e^{-\vp(\gamma_j')} \lesssim (\delta ^2 \sa _{\Gamma}^2)^{j+1}.
\]
Letting 
\[
\tilde F_n = \sum _{j=1} ^n  F_j,
\]
we have 
\[
||\tilde F_n|| \lesssim \frac{\sa _{\Gamma}}{1-\delta \sa _{\Gamma}},
\]
so by Proposition \ref{bergman-eucl} and Montel's Theorem, $\tilde F_n$ is a normal family.  Passing to a locally uniformly convergent subsequence, we obtain a function $F \in \sh ^2(\C, e^{-\vp} \omega _o)$ such that 
\[
F|_{\Gamma} = f \quad \text{and} \quad  ||F|| \le \frac{\sa _{\Gamma}}{1-\delta \sa _{\Gamma}},
\]
as desired.
\end{proof}

\begin{lem}\label{more-pos-wt-better}
If $\Gamma$ is an interpolation sequence for $\vp$, then for any $z \in \C$ and any $\ve > 0$, $\Gamma$ is an interpolation sequence for $\vp +\ve |\cdot -z|^2$, with interpolation constant independent of $z$, and at most a multiple of $\ve^{-3/2}$, with the multiple depending only on $\sa _{\Gamma}$ and the upper bound of $\Delta \vp$.  
\end{lem}

\begin{proof}
Since $\Gamma$ is an interpolation sequence, there exist functions $F_{\gamma} \in \sh ^2(\C , e^{-\vp}\omega _o)$ such that 
\[
F_{\gamma}(\mu) = \delta _{\gamma \mu}e^{\vp(\gamma)/2} \quad \text{and} \quad ||F_{\gamma}|| \le \sa _{\Gamma}.
\]
Let 
\[
\tilde F _{\gamma}(\zeta) := F_{\gamma}(\zeta) e^{\frac{\ve}{2}(2 (\bar \gamma -\bar z) (\zeta -z) -|\gamma -z|^2)}.
\]
Then 
\[
\tilde F _{\gamma} (\mu) = \delta _{\gamma\mu} e^{\frac{1}{2}( \vp (\gamma) +\ve |\gamma -z|^2)}
\]
and
\begin{eqnarray*}
\int _{\C} |\tilde F_{\gamma}(\zeta)|^2 e^{-\vp (\zeta)- \ve |\zeta -z|^2}\omega _o(\zeta) &=& \int _{\C} |F_{\gamma}(\zeta)|^2 e^{-\vp (\zeta)} e^{- \ve\left (|\zeta -z|^2 - 2 \re (\overline{\gamma -z})(\zeta-z) + |\gamma -z|^2\right )}\omega _o(\zeta)\\
&=& \int _{\C} |F_{\gamma}(\zeta)|^2 e^{-\vp (\zeta)} e^{- \ve |\zeta -\gamma|^2}\omega _o(\zeta) \\
&\le & C \sa _{\Gamma}^2 \int _{\C} e^{-\ve |\zeta-\gamma|^2}\omega _o (\zeta) = C \ve ^{-1} \sa _{\Gamma}^2,
\end{eqnarray*}
where $C$ depends only on the upper bound $M$ for $\Delta \vp$.  The inequality follows from Proposition \ref{bergman-eucl}(a), and then Proposition \ref{bergman-eucl}(a) implies that 
\[
|\tilde F_{\gamma}(\zeta)|^2e^{-\vp (\zeta) - \ve |\zeta-z|^2} \lesssim \sa _{\Gamma}^2 \ve ^{-1}.  
\]

Now let $f \in \ell ^2 (\Gamma , e^{-\vp - 2\ve |\cdot -z|^2})$.  Define 
\begin{eqnarray*}
F_f (\zeta)&:=& \sum _{\gamma \in \Gamma} f(\gamma) e^{-\frac{1}{2}(\vp(\gamma) + \ve |\gamma -z|^2)} e^{\ve( (\bar \gamma -\bar z) (\zeta -z) -|\gamma -z|^2)}\tilde F_{\gamma}(\zeta)\\
&=& \sum _{\gamma \in \Gamma} f(\gamma) e^{-\frac{1}{2}(\vp(\gamma) + 2\ve |\gamma -z|^2)} e^{\frac{\ve}{2}( 2 (\bar \gamma -\bar z) (\zeta -z) -|\gamma -z|^2)}\tilde F_{\gamma}(\zeta).
\end{eqnarray*}
Then 
\[
F_f|_{\Gamma} = f,
\]
and 
\begin{eqnarray*}
&& \int _{\C} |F_f(\zeta)|^2 e^{-\vp(\zeta) - 2\ve |\zeta -z|^2}\omega _o(\zeta) \\
&\le & \int _{\C}\left ( \sum _{\gamma \in \Gamma} |f(\gamma)| e^{-\frac{1}{2}(\vp(\gamma) + 2\ve |\gamma -z|^2)}e^{\frac{\ve}{2} ( 2 \re (\bar \gamma - \bar z)(\zeta -z) - |\gamma -z|^2)}  |\tilde F_{\gamma}(\zeta)| \right )^2 e^{-\vp(\zeta) - 2\ve |\zeta -z|^2}\omega _o(\zeta)\\
&=& \int _{\C}\left ( \sum _{\gamma \in \Gamma} |f(\gamma)| e^{-\frac{1}{2}(\vp(\gamma) + 2\ve |\gamma -z|^2)}e^{\frac{\ve}{2} (-|\zeta -z|^2 + 2 \re (\bar \gamma - \bar z)(\zeta -z) - |\gamma -z|^2)}  |\tilde F_{\gamma}(\zeta)| e^{-\frac{1}{2} (\vp(\zeta) +\ve |\zeta -z|^2)}\right )^2 \omega _o(\zeta)\\
&\le & \frac{\sa _{\Gamma}^2}{\ve} \int _{\C} \left ( \sum _{\gamma \in \Gamma} |f(\gamma)| e^{-\frac{1}{2}(\vp(\gamma) + 2\ve |\gamma -z|^2)}e^{-\frac{\ve}{2} |\zeta -\gamma|^2} \right )^2 \omega _o(\zeta)\\
&\le&  \frac{\sa _{\Gamma}^2}{\ve} \int _{\C} \left ( \sum _{\gamma \in \Gamma} |f(\gamma)|^2 e^{-(\vp(\gamma) + 2\ve |\gamma -z|^2)}e^{-\frac{\ve}{2} |\zeta -\gamma|^2}\right ) \left ( \sum _{\gamma \in \Gamma} e^{-\frac{\ve}{2} |\zeta -\gamma|^2} \right )\omega _o(\zeta).
\end{eqnarray*}
Since $\Gamma$ is uniformly separated, the second sum converges uniformly for any $\ve > 0$ to a function that is bounded by $\ve ^{-1}$ times a constant that depends on the separation radius of $\Gamma$.  We therefore have 
\[
\int _{\C} |F_f(\zeta)|^2 e^{-\vp(\zeta) - 2\ve |\zeta -z|^2}\omega _o(\zeta) \le \frac{C}{\ve ^3} \sum _{\gamma \in \Gamma}  |f(\gamma)|^2 e^{-(\vp(\gamma) + 2\ve |\gamma -z|^2)}
\]
where $C$ depends only on $M$ and $\sa _{\Gamma}$.  This completes the proof.
\end{proof}

\begin{prop}\label{add-eucl}
Let $\Gamma$ be an interpolation sequence, and let $z \in \C - \Gamma$ satisfy ${\rm dist}(z,\Gamma) > \delta$.  Then the sequence $\Gamma_z  := \Gamma \cup \{z\}$ is an interpolation sequence for the weight $\psi (\zeta) := \vp (\zeta) + \ve |\zeta -z|^2$, and its interpolation constant is bounded above by a constant of the form $K/(\delta \ve ^{5/2})$, where $K$ depends only on $M$ and $\Gamma$.
\end{prop}

\begin{proof}
It suffices to show that there exists $F \in \sh ^2(\C, e^{-\psi}\omega _o)$ satisfying 
\[
F(z) = e^{\vp(z)/2} \quad \text{and} \quad F|_{\Gamma} \equiv 0
\]
with appropriate norm bounds.  To this end, write 
\[
\chi (\zeta) := \vp (\zeta) + \frac{\ve}{2} |\zeta -z|^2.
\]
Lemma \ref{unif-1pt-interp-eucl} provides us with a function $G \in \sh ^2(\C, e^{-\chi}\omega _o)$ such that 
\[
G(z) = e^{\vp (z)/2} \quad \text{and} \quad \int _{\C} |G|^2 e^{-\chi} \omega _o \le \frac{C}{\ve},
\]
where $C$ only depends on the upper bound of $\Delta \vp$, and not on $z$ or $\Gamma$.  Observe that by Corollary \ref{Bergman-sums-eucl}(a)
\[
\sum _{\gamma \in \Gamma} \frac{|G(\gamma)|^2e^{-\chi(\gamma)}}{|z-\gamma|^2} \lesssim \frac{1}{\delta ^2}\sum _{\gamma \in \Gamma} \int _{D_{R_{\Gamma}}(\gamma)} |G|^2e^{-\chi} \omega _o \lesssim \frac{1}{\delta ^2 \ve}.
\]
Since $\Gamma$ is an interpolation sequence for $\sh ^2 (\C , e^{-\vp}\omega _o)$, it is also an interpolation sequence for $\sh ^2 (\C , e^{-\chi}\omega _o)$.  Thus there exists $H \in \sh ^2 (\C , e^{-\chi}\omega _o)$ such that 
\[
H(\gamma ) = \frac{G(\gamma)}{z-\gamma} \quad \text{and} \quad \int _{\C} |H|^2 e^{-\chi} \omega _o \lesssim \frac{\sa _{\Gamma}^2}{\delta ^2 \ve^4}.
\]
(We have used the fact that the interpolation constant with respect to $\chi$ is controlled by $\ve ^{-3/2}$ times the interpolation constant with respect to $\vp$.)  Let $F \in \co (\C)$ be defined by 
\[
F (\zeta) := G(\zeta) - (\zeta - z)H(\zeta).
\]
Then 
\[
F(z) = G(z) = e^{\vp(z)/2}, \quad \text{and} \quad F(\gamma) = G(\gamma) - (\gamma -z)H(\gamma) = 0
\]
for all $\gamma \in \Gamma$.  Finally, 
\begin{eqnarray*}
\left ( \int _{\C} |F|^2 e^{-\psi} \omega _o \right )^{1/2} &\le& \left ( \int _{\C} |G|^2 e^{-\psi} \omega _o \right )^{1/2} + \left ( \int _{\C} |H(\zeta)|^2|\zeta-z|^2 e^{-\psi(\zeta)} \omega _o(\zeta) \right )^{1/2}\\
& \le &  \left ( \int _{\C} |G|^2 e^{-\chi} \omega _o \right )^{1/2} + \frac{1}{\sqrt{\ve}} \left ( \int _{\C} |H(\zeta)|^2e^{-\chi(\zeta)}\omega _o(\zeta) \right )^{1/2}\\
&\le&  \frac{C(1 + \sa _{\Gamma})}{\delta \ve^{5/2}},
\end{eqnarray*}
as desired.
\end{proof}

\subsubsection{\sf Estimating the density of an interpolation sequence}

We are going to estimate the density of $\Gamma$ in two exhaustive, mutually exclusive cases.  In the first case, we estimate the density at a point $z$ of distance at most $\min (\sa _{\Gamma}^{-1}, R^o_{\Gamma})$ to $\Gamma$, and in the second case, when $z$ lies at least a distance $\min (\sa _{\Gamma}^{-1}, R^o_{\Gamma})$ from $\Gamma$.

In the first case, by Proposition \ref{jiggle-eucl} we may replace the nearest point $\gamma \in \Gamma$ by $z$, and still obtain an interpolation sequence $\Gamma '$, with slightly worse interpolation constant.  Since $\Gamma '$ is an interpolation sequence, we can find a function $F \in \sh ^2(\C, e^{-\vp}\omega _o)$ that vanishes on $\Gamma '-\{z\}$ and satisfies 
\[
F(z) = e^{\vp(z)/2} \quad \text{and} \quad ||F|| \lesssim \sa _{\Gamma}.
\]
By Jensen's Formula \ref{Jensen-eucl} applied to $F$, we have 
\[
\int_{\A^o_r(z)} \log \frac{r^2}{|\zeta -z|^2} \delta_{\Gamma}  \le \frac{1}{2\pi} \int_{D^o_r(z)} \log \frac{r^2}{|\zeta -z|^2} \Delta \vp (\zeta) + \frac{1}{2\pi} \int _{\di D^o_r(z)} \log (|F|^2e^{-\vp}) d\theta_z 
\]
By Proposition \ref{bergman-eucl}, we have 
\[
\int_{\A^o_r(z)} \log \frac{r^2}{|\zeta -z|^2} \delta_{\Gamma}  \le \frac{1}{2\pi} \int_{D^o_r(z)} \log \frac{r^2}{|\zeta -z|^2} \Delta \vp (\zeta) + C, 
\]
where $C$ is independent of $z$ and $r$.

Turning to the second case, by Proposition \ref{add-eucl} the sequence $\Gamma _z := \Gamma \cup \{z\}$ is an interpolation sequence for $\psi$, with interpolation constant at most $\frac{K}{\ve ^{\alpha}}$.  We can thus find $F \in \sh ^2(\C , e^{-\psi} \omega_o)$ such that 
\[
F(z) = e^{\vp(z)/2}, \quad F|_{\Gamma} \equiv 0 \quad \text{and} \quad |F|^2e^{-\psi} \lesssim ||F|| \lesssim \frac{K^2}{\ve ^{2\alpha}}.
\]
Again by Jensen's Formula and Proposition \ref{bergman-eucl}, we have 
\begin{equation}\label{jensen-est}
\int_{\A^o_r(z)} \log \frac{r^2}{|\zeta -z|^2} \delta_{\Gamma}  \le \frac{1}{2\pi} \int_{D^o_r(z)} \log \frac{r^2}{|\zeta -z|^2} (\Delta \vp (\zeta) + \ve \omega _o(\zeta)) - C\log \ve. 
\end{equation}
Thus in both cases, we have the estimate \eqref{jensen-est}.

Now consider the sequence obtained by moving all the points of $\Gamma$ a small distance $\delta$ toward $z$.  By Proposition \ref{jiggle-eucl}, this new sequence is an interpolation sequence as well.  Applying Jensen's formula to this modified sequence and making the change of variables $\zeta \mapsto \frac{r(\zeta-z)}{r+\delta} +z$, we have the estimate
\[
\int_{2\le |\zeta-z| \le r} \log \frac{r^2}{|\zeta -z|^2} \delta_{\Gamma}  \le \frac{1}{2\pi(1+\frac{\delta}{r})} \int_{D^o_{r}(z)} \log \frac{r^2}{|\zeta -z|^2} (\Delta \vp (\zeta) + \ve \omega _o(\zeta)) - C\log \ve. 
\]
Notice that, up to this point, we have not needed the lower bound on $\Delta \vp_r$.  But to control the enormous constant $-C\log \ve$, we need this hypothesis.  Indeed, let us choose $\ve = r^{-2}$.  Then we have 
\begin{eqnarray*}
\int_{2\le |\zeta-z| \le r} \log \frac{r^2}{|\zeta -z|^2} \delta_{\Gamma}  &\le& \frac{1}{2\pi} \int_{D^o_{r}(z)} \log \frac{r^2}{|\zeta -z|^2} \Delta \vp (\zeta)  - \frac{(m-r^{-2}) \delta r^2}{1+\frac{\delta}{r}}+ 2C\log r
\end{eqnarray*}
It follows that for sufficiently large $r$, 
\[
D^+_{\vp}(\Gamma) \le 1-\frac{m\delta }{2M} < 1.
\]
This completes the proof of Theorem \ref{necess-eucl}, and thus of Theorem \ref{char-eucl}.
\qed

\section{Interpolation in $(\C^*, \omega _c)$}\label{cyl-section}

\subsection{Cylindrical distance, covered means, and cover density} 
We make use of the {\it cylindrical distance}, i.e., the geodesic distance $d_c$ of the cylindrical metric $\omega _c$.  Since the universal covering map 
\[
\fp :\C \to \C^*; \zeta \mapsto e^{\zeta}
\]
is a local isometry, and the deck group of $\fp$ is generated by the translation $z\mapsto z+2\pi \ii$, 
\begin{enumerate}
\item[(i)] the distance between two points $\zeta, \eta \in \C^*$ is 
\[
d_c(\zeta,\eta) = \sqrt{(\log |\zeta/\eta|)^2 + (\arg (\zeta/\eta) )^2},
\]
where $\arg$ is the argument starting from the ray that is orthogonal to any half-space containing the points $\eta$ and $\zeta$ and whose boundary contains the origin (so that in particular, $\arg (\zeta/\eta) \in [0,\pi]$), and 
\item[(ii)] a sequence $\Gamma \subset \C^*$ is uniformly separated in the cylindrical distance if and only if the inverse image $\fp ^{-1}(\Gamma)$ is uniformly separated in the Euclidean distance.
\end{enumerate}

\noi By analogy with the case of Euclidean space, we define the {\it separation radius} 
\[
R^c_{\Gamma} := \frac{1}{2}  \inf \{ d_c (\gamma _1, \gamma _2)\ ;\ \gamma _1, \gamma _2\in \Gamma,\ \gamma _1 \neq \gamma _2 \}
\]
of $\Gamma$, so that again $\Gamma$ is uniformly separated if and only if $R^c_{\Gamma} > 0$.

Let $\vp \in L^{1}_{\ell oc} (\C^*)$.  Using the notation \eqref{log-avg-defn}, consider the function 
\[
(\fp ^* \vp)_r(z), \quad z \in \C.
\]
Since $\fp (z+2\pi \ii) = \fp (z)$, 
\[
(\fp ^*\vp)_r(z+2\pi \ii) = (\fp ^*\vp)_r(z),
\]
and thus it follows that 
\[
(\fp ^*\vp)_r (z) = \mu(\vp)_r (e^z)
\]
for some uniquely determined function $\mu(\vp)_r : \C^* \to \R$.
\begin{defn}\label{covered-mean-defn}
The function $\mu(\vp)_r$ is called the {\it covered mean} of $\vp$ (over the disk of radius $r$).
\red
\end{defn}

Observe that if $\vp$ is subharmonic, then so is $\fp ^* \vp$.  Thus for subharmonic $\vp$, 
\[
\fp ^*\vp \le (\fp ^*\vp)_r, \quad \text{and therefore} \quad \vp \le \mu(\vp)_r.
\]

Finally, we turn to the cover density. 

\begin{defn}
Let $\vp :\C^* \to [-\infty,\infty)$ be subharmonic.  The {\it cover density} of a sequence $\Gamma \subset \C^*$ is
\[
\tilde D^+_{\vp}(\Gamma) := D^+_{\fp ^*\vp} (\tilde \Gamma),
\]
where $\fp : \C\to \C^*$ is the (universal covering) exponential map and $\tilde \Gamma = \fp ^{-1}(\Gamma)$.
\red
\end{defn}

\subsection{The main result for $(\C^*, \omega _c)$}
The main interpolation result for $(\C^*, \omega _c)$ can now be stated.

\begin{d-thm}\label{char-cyl}
Let $\vp \in \sC ^2 (\C^*)$ be a weight function satisfying 
\begin{equation}\label{cyl-curv-bd}
0< m\omega _c \le \Delta \vp \le M\omega _c, 
\end{equation}
and let $\Gamma \subset \C^*$ be a closed discrete subset.  Denote by  
\[
\sr_{\Gamma} :\sh ^2 (\C^* , e^{-\vp}\omega _c ) \to \ell ^2 (\Gamma, e^{-\vp})
\]
the restriction map.  If 
\begin{enumerate}
\item[(i+)] $\Gamma$ is uniformly  separated with respect to the cylindrical distance, and 
\item[(ii+)] $\tilde D^+_{\vp}(\Gamma) < 1$, 
\end{enumerate}
then $\sr _{\Gamma}$ is surjective.  Conversely, if $\sr _{\Gamma}$ is surjective, then 
\begin{enumerate}
\item[(i-)] $\Gamma$ is uniformly  separated with respect to the cylindrical distance, and 
\item[(ii-)] $\tilde D^+_{\vp}(\Gamma) \le1$. 
\end{enumerate}
\end{d-thm}

\begin{s-rmk}
Even if we assume only that $\vp \in L^1_{\ell oc}(\C^*)$, standard regularity theory and condition \eqref{cyl-curv-bd} imply that $\vp \in \sC ^{1,\alpha}$.
\red
\end{s-rmk}

\subsection{Sufficiency}

We begin with the analogue of Theorem \ref{s-suff-eucl} for $(\C^*, \omega _c)$.

\begin{d-thm}[Strong sufficiency: Cylindrical case]\label{s-suff-cyl}
Let $\vp :\C^* \to [-\infty, \infty)$ be any subharmonic weight.  Assume that $\Gamma \subset \C^*$ is uniformly separated with respect to the cylindrical distance, and that
\[
\Delta \vp  \ge \alpha \Delta \mu(\log |T|)_r 
\]
for some $\alpha > 1$.  Then the restriction $\sr_{\Gamma} :\sh ^2 (\C^* , e^{-\vp}\omega _c ) \to \ell ^2 (\Gamma, e^{-\vp})$ is surjective.
\end{d-thm}

As in the proof of Theorem \ref{s-suff-eucl}, we begin by applying the $L^2$ Extension Theorem, namely, Theorem \ref{ot-basic}.  In that theorem, set $(X,\omega) = (\C^*, \omega _c)$, fix a function $T \in \co (\C^*)$ with ${\rm Ord}(T)=\Gamma$, and take $\lambda := \mu(\log|\fp^*T|^2)_r$.  Then $|T|^2e^{-\lambda} \le 1$, and the curvature conditions of Theorem \ref{ot-basic} mean exactly that $\tilde D^+_{\vp}(\Gamma) < 1$.  We therefore have the following result.

\begin{d-thm}
Let $\vp$ be a plurisubharmonic function on $\C^*$, and let $\Gamma \subset \C^*$ be any closed discrete subset satisfying $\tilde D^+_{\vp}(\Gamma) < 1$.  Then for any $f : \Gamma \to \C$ satisfying 
\[
\sum _{\gamma \in \Gamma} \frac{|f(\gamma)|^2e^{-\vp(\gamma)}}{|dT(\gamma)|^2_{\omega _c} e^{-\mu(\log |\fp^*T|^2)_r(\gamma)}} < +\infty
\]
there exists $F \in \sh ^2 (\C^*, e^{-\vp}\omega _c)$ such that 
\[
F|_{\Gamma} = f \quad \text{and} \quad \int _{\C^*} |F|^2e^{-\vp} \omega _c \le \frac{24\pi}{1-D^+_{\vp}(\Gamma)} \sum _{\gamma \in \Gamma} \frac{|f(\gamma)|^2e^{-\vp(\gamma)}}{|dT(\gamma)|^2_{\omega _c} e^{-\mu(\log|\fp^*T|^2)_r(\gamma)}}.
\]
\end{d-thm}

The proof of Theorem \ref{s-suff-cyl} then follows from the following result.

\begin{prop}\label{unif-sep-lower-bd-cyl}
Let $\Gamma \subset \C^*$ be a closed discrete subset.  Then $\Gamma$ is uniformly separated with respect to the cylindrical distance if and only if for any $r > 1$ there exists $C_r>0$ such that 
\[
\inf _{\gamma \in \Gamma} |dT(\gamma)|^2_{\omega_c} e^{-\mu(\log |\fp^*T|^2)_r(\gamma)} \ge C_r.
\]
\end{prop}

\begin{proof}
Recall that $\Gamma$ is uniformly separated with respect to the cylindrical distance if and only if $\tilde \Gamma \subset \C$ is uniformly separated with respect to the Euclidean distance.  The result therefore follows from its Euclidean analogue, Proposition \ref{denom-bds-eucl}, and the definition of the covered mean $\mu (\vp)_r$.
\end{proof}

Finally, if we replace of $\vp$ by $\mu (\vp)_r$,  Theorem \ref{s-suff-cyl} implies the `if' direction of Theorem \ref{char-cyl}.

\subsection{Necessity}

As in the Euclidean case, we now turn our attention to the necessity of the conditions of Theorem \ref{char-cyl}.  That is to say, we shall prove the following theorem.

\begin{d-thm}\label{necess-cyl}
Let $\vp \in \sC ^2 (\C^*)$ be a weight function satisfying 
\[
m\omega _c \le\Delta \vp \le M \omega _c
\]
for some positive constants $m$ and $M$, and let $\Gamma \subset \C^*$ be a closed discrete subset.  If 
\[
\sr _{\Gamma} :\sh ^2(\C^*, e^{-\vp}\omega _c ) \to \ell ^2 (\Gamma, e^{-\vp})
\]
is surjective, then $\Gamma$ is uniformly separated with respect to the cylindrical distance, and $\tilde D^+_{\vp} (\Gamma) \le 1$.
\end{d-thm}

\subsubsection{\sf The interpolation constant}

As in the Euclidean case, if $\Gamma \subset \C^*$ is an interpolation sequence, then the restriction operator 
\[
\sr _{\Gamma} : \sh ^2(\C ^*, e^{-\vp}\omega _o) \to \ell ^2 (\Gamma, e^{-\vp})
\]
has bounded inverses, and the extension operator of minimal norm 
\[
\se _{\Gamma} : \ell ^2 (\Gamma, e^{-\vp}) \to {\rm Kernel}(\sr _{\Gamma})^{\perp} \subset \sh ^2(\C ^*, e^{-\vp}\omega _o)
\]
is one such operator.  Moreover, the interpolation constant 
\[
\sa _{\Gamma} := \inf \{ A\ ;\ \exists E : \ell ^2 (\Gamma, e^{-\vp}) \to  \sh ^2(\C ^*, e^{-\vp}\omega _o) \text{ with }\sr _{\Gamma}E={\rm Id}\text{ and }||E||\le A\}
\]
is precisely the norm of $\se _{\Gamma}$.

\subsubsection{\sf Necessity of Uniform Separation}

Suppose $\Gamma \subset \C^*$ is an interpolation sequence.  We show that $\tilde \Gamma \subset \C$ is uniformly separated in the Euclidean distance.  For each $t\in \R$, denote by $S_t \subset \C$ the set of all points $z$ such that 
\[
t \le \im z < t+2\pi.
\]
For any $t \in \R$, the strip $S_t$ is a fundamental domain of the universal covering map $\fp (z) = e^z$.

Fix two points $\gamma _1, \gamma _2 \in \Gamma$.  We choose points $\tilde \gamma _1 \in \fp ^{-1}(\gamma_1)$ and $\tilde \gamma _2 \in \fp ^{-1}(\gamma_2)$, and a real number $t$,  such that the Euclidean distance between $\tilde \gamma _1$ and $\tilde \gamma _2$ is the cylindrical distance between $\gamma _1$ and $\gamma_2$, and which is equal to the length of the straight line in $S_t$ connecting  $\tilde \gamma _1$ and $\tilde \gamma _2$.  After choosing and appropriate branch of the logarithm, we may write $\tilde \gamma _i = \log \gamma _i$, $i=1,2$.  We can assume that the straight line joining $\log \gamma _1$ and $\log \gamma _2$ has Euclidean length at most $\pi$; otherwise the two points are at least a distance $\pi$ apart, and there is nothing to prove.  We define the $f \in \ell ^2(\Gamma , e^{-\vp})$ by 
\[
f(\gamma _1) = e^{\vp(\gamma _1)/2} \quad \text{and} \quad f(\mu) = 0 \text{ for all }\gamma \in \Gamma - \{\gamma _1\}.
\]
Since $||f||^2_{\ell ^2(\Gamma , e^{-\vp})} = 1$ and $\Gamma$ is an interpolation sequence, there is a function 
\[
F \in \sh ^2(\C^*, e^{-\vp}\omega _c)
\]
such that 
\[
|F(\gamma _1)|^2e^{-\vp(\gamma_1)} = 1, \quad F(\gamma _2)= 0 \quad \text{and} \quad \int _{\C^*} |F|^2 e^{-\vp} \omega _c \le \sa _{\Gamma}^2.
\]

Now define 
\[
\tilde F = \fp ^* F \quad \text{and} \quad \tilde \vp = \fp ^*\vp.
\]
Then, with $U :=  (S_t-2\pi \ii) \cup S_t \cup (S_t+2\pi \ii)$, 
\[
|\tilde F(\log \gamma _1)|^2e^{-\tilde \vp(\log \gamma_1)} = 1, \quad \tilde F(\log \gamma _2)= 0 \quad \text{and} \quad \int _U |\tilde F|^2 e^{-\tilde \vp} \omega _o \le 3\sa ^2_{\Gamma}.
\]
By Proposition \ref{bergman-eucl}(b) with $r = 2\pi$, we conclude that 
\[
\frac{1}{{\rm dist}_c (\gamma _1, \gamma _2)}= \frac{1}{|\log \frac{\gamma _1}{\gamma_2}|} = \frac{|\tilde F(\log \gamma_1)|^2e^{-\tilde \vp(\log \gamma _1)} - |\tilde F(\log \gamma_2)|^2e^{-\tilde \vp(\log \gamma_2)}}{|\log \gamma _1 - \log \gamma_2|} \le  C
\]
for some constant $C$ independent of $\gamma _1$ and $\gamma_2$.  Thus $\Gamma$ is uniformly separated in the cylindrical metric.

\subsubsection{\sf Uniform interpolation at a point}
\begin{lem}\label{unif-1pt-interp-cyl}
Let $\vp \in \sC ^2(\C^*)$ be a weight function satisfying 
\[
\Delta \vp \ge a\omega _c
\]
for some positive constant $a$.  Then there exists a constant $C>0$ such that for any $z\in \C^*$ there is a function $F\in \sh ^2(\C^*, e^{-\vp}\omega _c)$ satisfying 
\[
|F(z)|^2e^{-\vp(z)} = 1 \quad \text{and} \quad \int _{\C^*} |F|^2e^{-\vp} \omega _c \le C.
\]
\end{lem}

\begin{proof}
We adapt the idea of the proof of Lemma \ref{unif-1pt-interp-eucl}.  Consider the holomorphic function $T_z(\zeta) = \zeta -z$ and the function $\lambda_z := \mu (\log |T_z|^2)_r :\C^* \to \R$.  Observe that since $\Delta \vp \ge a \omega _c$, for any $\delta > 0$, we can find $r >>0$ such that 
\[
\ii(\di \dbar  \vp + {\rm Ricci}(\omega _c)) = \Delta \vp  \ge (1+\delta) \lambda_z.
\]
We can therefore apply Theorem \ref{ot-basic} to obtain a function $F \in \co (\C^*)$ such that 
\[
F(z) = e^{\vp(z)/2} \quad \text{and} \quad \int _{\C^*} |F|^2e^{-\vp} \omega _c \le \frac{C}{|dT_z(z)|^2_{\omega _c}e^{-\lambda _z(z)}} ,
\]
with $C$ independent of $z$.  Since a sequence consisting of a single point is uniformly separated, an application of Proposition \ref{unif-sep-lower-bd-cyl}, especially in view of Remark \ref{uniformity-of-denom-bd}, completes the proof.
\end{proof}

\subsubsection{\sf Perturbation of interpolation sequences}

\begin{prop}\label{jiggle-cyl}
Let $\Gamma \subset \C^*$ be an interpolation sequence with separation radius $R^c_{\Gamma}$, enumerated as $\Gamma = \{ \gamma _1, \gamma_2, ... \}$, let $\sa _{\Gamma}$ be the interpolation constant of $\Gamma$.  Suppose $\Gamma '\subset \C^*$ is another sequence, such that there exists a constant $\delta \in (0, \min (\sa^{-1} _{\Gamma}, R^c_{\Gamma}))$, and an enumeration $\Gamma ' = \{ \gamma '_1, \gamma '_2, ... \}$ so that 
\[
\sup _{i\in \N} d_c(\gamma _i,\gamma '_i) \le \delta ^2.
\]
Then $\Gamma '$ is also an interpolation sequence, and its interpolation constant is at most 
\[
C\frac{\sa _{\Gamma}}{1-\delta \sa _{\Gamma}},
\]
where $C$ is independent of $\Gamma$.
\end{prop}

\begin{proof}
First observe that if $F \in \sh ^2(\C^*, e^{-\vp}\omega_c)$ then 
\begin{equation}\label{comparison-cyl}
\sum _{j=1} ^{\infty} \left | |F(\gamma_j)|^2e^{-\vp(\gamma_j)} - |F(\gamma'_j)|^2e^{-\vp(\gamma'_j)}\right | \lesssim \delta ^2 \int _{\C^*} |F|^2e^{-\vp} \omega _c.
\end{equation}
To obtain this estimate for $F$, we must lift small disks containing the points of $\Gamma$ to the universal cover and use Corollary \ref{Bergman-sums-eucl}(b).  We can carry out this step with disks of a uniform radius because we have already shown that an interpolation sequence is uniformly separated with respect to the cylindrical distance.

The rest of the proof is the same as the Euclidean case, established previously as Proposition \ref{jiggle-eucl}.
\end{proof}

\begin{lem}\label{1-pt-estimate-cyl}
Let $a > 0$, let $\delta \in (0,1/2)$, and let $x \in \C^*$.  
\begin{enumerate}
\item[(i)] $|1-x|^2e^{-a(\log|x|)^2} \le 4e^{a^{-1}}.$
\item[(ii)] If $d_c(x,1) \ge \delta$ then $|1-x|^2 \ge C_{\delta}$, where 
\[
\lim_{\delta \to 0}\delta ^{-2}C_{\delta} =1.
\]
\end{enumerate}
\end{lem}

\begin{proof}
(i) Let $r = \log |x|$ and $\theta = \arg x$.  Then 
\[
|1-x|e^{-a(\log |x|)^2/2} \le (1+e^r)e^{-ar^2/2} \le 1 + e^{-\frac{a}{2}(r^2 -\frac{2r}{a})} = 1 + e^{\frac{1}{2a}-\frac{a}{2}(r -\frac{1}{a})^2} \le 1+e^{1/2a}.
\]
Taking squares, we have $|1-x|^2e^{-a(\log|x|)^2} \le 1+2e^{1/2a}+e^{1/a} \le 4e^{1/a}$.

\noi (ii) If we write $x=e^{s+\ii t}$ then $d_c(x,1) = s^2+t^2$, while $|x-1|^2 = e^{2s}+1 - 2e^s \cos t$.  Taylor's Theorem shows that for $s$ and $t$ small, $e^{2s}+1 - 2e^s \cos t = s^2 + t^2 + o(s^2+t^2)$.
\end{proof}

\begin{prop}\label{add-cyl}
Assume $m\omega _c \le \Delta \vp \le M \omega _c$ for some positive constants $m$ and $M$.  Let $\Gamma$ be an interpolation sequence, and let $z \in \C^* - \Gamma$ satisfy ${\rm dist}_c(z,\Gamma) \ge \delta> 0$.  Then for $\ve >0$ the sequence $\Gamma_z  := \Gamma \cup \{z\}$ is an interpolation sequence for $\sh ^2(\C^*, e^{-(\vp + \ve (\log |\cdot/z|)^2}\omega _c)$, and its interpolation constant is bounded above by some constant $K/ \ve $, where $K$ depends only on $M$, $\Gamma$ and $\delta$, and in particular, not on $z$.
\end{prop}

\begin{proof}
Write 
\[
\psi _z := \vp - \frac{m}{2} ( \log |\zeta/z|)^2 \quad \text{and} \quad \eta _z := \vp + \ve ( \log |\zeta/z|)^2.
\]
Since $\eta _z (z) = \vp(z)$, it suffices to show that there exists $F \in \sh ^2(\C^*, e^{-\eta _z}\omega _c)$ satisfying 
\[
F(z) = e^{\vp(z)/2} \quad \text{and} \quad F|_{\Gamma} \equiv 0
\]
with appropriate norm bounds.  To this end, since $\Delta \psi _z \ge \frac{m}{2} \omega _c$, Proposition \ref{unif-1pt-interp-cyl} provides us with a function $G \in \sh ^2(\C^*, e^{-\psi_z}\omega _c)$ such that 
\[
G(z) = e^{\vp (z)/2} \quad \text{and} \quad \int _{\C^*} |G|^2 e^{-\psi_z} \omega _c \le C,
\]
where $C$ does not depend on $z$ or $\Gamma$.  

Now, by (ii) of Lemma \ref{1-pt-estimate-cyl} and Corollary \ref{Bergman-sums-eucl}(a) (the latter of which can be applied after passing to the universal cover as in the proof of Proposition \ref{jiggle-cyl}) we have the estimate  
\[
\sum _{\gamma \in \Gamma} \frac{|G(\gamma)|^2e^{-\vp(\gamma)}}{|1-\frac{\gamma}{z}|^2} \lesssim \frac{1}{\delta ^2}\sum _{\gamma \in \Gamma} \int _{D^c_{R_{\Gamma}}(\gamma)} |G|^2e^{-\psi_z} \omega _c \lesssim \frac{1}{\delta ^2}.
\]
Since $\Gamma$ is an interpolation sequence for $\sh ^2 (\C^* , e^{-\vp}\omega _c)$, there exists $H \in \sh ^2 (\C^* , e^{-\vp}\omega _c)$ such that 
\[
H(\gamma ) = \frac{G(\gamma)}{1-\frac{\gamma}{z}}, \ \gamma \in \Gamma, \quad \text{and} \quad \int _{\C^*} |H|^2 e^{-\vp} \omega _c \lesssim \frac{\sa _{\Gamma}^2}{\delta ^2}.
\]
Let $F \in \co (\C^*)$ be defined by 
\[
F (\zeta) := G(\zeta) - \left ( 1 - \frac{\zeta}{z} \right )H(\zeta).
\]
Then 
\[
|F(z)|^2e^{-\vp(z)} = |G(z)|^2e^{-\vp(z)}=1, \quad \text{and} \quad F(\gamma) = G(\gamma) - \left (1- \frac{\gamma}{z}\right )H(\gamma) = 0
\]
for all $\gamma \in \Gamma$.  Finally, using (i) of Lemma \ref{1-pt-estimate-cyl}, we estimate that 
\begin{eqnarray*}
\left ( \int _{\C^*} |F|^2 e^{-\eta_z} \omega _c \right )^{1/2} &\le& \left ( \int _{\C^*} |G|^2 e^{-\eta_z} \omega _c \right )^{1/2} + \left ( \int _{\C^*} |H(\zeta)|^2|1-\tfrac{\zeta}{z}|^2 e^{-\eta_z(\zeta)} \omega _c(\zeta) \right )^{1/2}\\
& \le &  \left ( \int _{\C} |G|^2 e^{-\psi _z} \omega _c \right )^{1/2} + \frac{1}{\sqrt{\ve}} \left ( \int _{\C} |H(\zeta)|^2e^{-\vp(\zeta)}\omega _o(\zeta) \right )^{1/2}\\
&\le&  \frac{C(1 + \sa _{\Gamma})}{\delta \ve},
\end{eqnarray*}
as desired.
\end{proof}

\subsubsection{\sf Estimating the density of an interpolation sequence}

As in the Euclidean case, we will estimate the cover density of $\Gamma$ in two exhaustive, mutually exclusive cases.  In the first case, we estimate the cover density at a point $z$ of cylindrical distance at most the square of $\min (\sa _{\gamma}^{-1}, R^c_{\Gamma})$ to $\Gamma$, and in the second case, when the cylindrical distance from $z$ to $\Gamma$ is at least the square of $\min (\sa _{\Gamma}^{-1}, R^c_{\Gamma})$.

In the first case, if ${\rm dist}_c(z,\Gamma)\le \delta^2$ for some $\delta < \min (\sa _{\gamma}^{-1}, R^c_{\Gamma})$, by Proposition \ref{jiggle-cyl} we may replace the nearest point $\gamma \in \Gamma$ by $z$, and still obtain an interpolation sequence $\Gamma '$, with a possibly slightly worse interpolation constant that is at most $C \frac{\sa _{\Gamma}}{1-\delta \sa _{\Gamma}}$.  Since $\Gamma '$ is an interpolation sequence, we can find a function $F \in \sh ^2(\C^*, e^{-\vp}\omega _c)$ that vanishes on $\Gamma '-\{z\}$ and satisfies 
\[
F(z) = e^{\vp(z)/2} \quad \text{and} \quad ||F|| \lesssim \frac{\sa _{\Gamma}}{1-\delta \sa_{\Gamma}}.
\]
Now write 
\[
\tilde F := \fp ^*F,\quad \tilde \vp := \fp ^*\vp,\quad \tilde \Gamma := \fp ^{-1}(\Gamma) \text{, \quad and \quad }\tilde \Gamma _z := \fp ^{-1}(\Gamma - \{z\}),
\]
where $\fp : \C \to \C^*$ is the universal cover.  By Jensen's Formula \ref{Jensen-eucl} applied to $\tilde F$, for any $x \in \fp ^{-1}(z)$ we have 
\[
\int_{\A^o_r(x)} \log \frac{r^2}{|\zeta -x|^2} \delta_{\tilde \Gamma_z}  \le \frac{1}{2\pi} \int_{D^o_r(x)} \log \frac{r^2}{|\zeta -x|^2} \Delta \tilde \vp (\zeta) + \frac{1}{2\pi} \int _{\di D^o_r(x)} \log (|\tilde F|^2e^{-\vp}) d\theta_x.
\]
By Proposition \ref{bergman-eucl}, we have 
\[
\int_{\A^o_r(x)} \log \frac{r^2}{|\zeta -x|^2} \delta_{\tilde \Gamma}  \le \frac{1}{2\pi} \int_{D^o_r(x)} \log \frac{r^2}{|\zeta -x|^2} \Delta \tilde \vp (\zeta) + C, 
\]
where $C$ is independent of $z$ and $r$.  

Turning to the second case, by Proposition \ref{add-cyl} the sequence $\Gamma _z := \Gamma \cup \{z\}$ is an interpolation sequence for $\eta _z = \vp + \ve (\log |\zeta/z|)^2$, with interpolation constant at most ${K}/{\ve}$.  We can thus find $F \in \sh ^2(\C^* , e^{-\eta _z} \omega_c)$ such that 
\[
F(z) = e^{\vp(z)/2}, \quad F|_{\Gamma} \equiv 0 \quad \text{and} \quad |F|^2e^{-\eta_z} \lesssim ||F|| \lesssim \frac{K^2}{\ve ^{2}}.
\]
Again by Jensen's Formula and Proposition \ref{bergman-eucl}, we have 
\begin{equation}\label{jensen-est-cyl}
\int_{\A^o_r(x)} \log \frac{r^2}{|\zeta -x|^2} \delta_{\tilde \Gamma}  \le \frac{1}{2\pi} \int_{D^o_r(x)} \log \frac{r^2}{|\zeta -x|^2} (\Delta \tilde \vp (\zeta) + \ve \omega _o(\zeta)) -  C\log \ve. 
\end{equation}
Thus in both cases, we have the estimate \eqref{jensen-est-cyl}.  

Since 
\[
\lim_{r \to \infty} \frac{1}{2\pi} \int_{D^o_r(x)} \log \frac{r^2}{|\zeta -x|^2} \Delta \tilde \vp (\zeta)  = +\infty
\]
and, in view of \eqref{cyl-curv-bd},
\[
\lim_{r \to \infty}\frac{\int_{D^o_r(x)} \log \frac{r^2}{|\zeta -x|^2} \omega _o(\zeta)}{\int_{D^o_r(x)} \log \frac{r^2}{|\zeta -x|^2} \Delta \tilde \vp (\zeta)} \le \frac{1}{m}, 
\]
the estimate \eqref{jensen-est-cyl} implies that $\tilde D ^+_{\vp}(\Gamma) \le 1+\ve/m$.  Since $\ve >0$ is arbitrary, $D^+_{\vp}(\Gamma) \le 1$.  This completes the proof of Theorem \ref{necess-cyl}, and thus of Theorem \ref{char-cyl}.
\qed

\section{Interpolation on asymptotically flat finite Riemann surfaces}\label{main-section}

We are now ready to turn to the proof of Theorem \ref{main}.  Let us fix once and for all a compact set $K \relcomp X$ with smooth codimension-$1$ boundary, disjoint open sets $U_1,...,U_n, U_{n+1},...,U_{n+m} \subset X-K$ such that  
\[
K \cup \bigcup _{j=1} ^{n+m} U_j = X,
\]
and biholomorphic maps $F_j : \C - D_j \to U_j$, $1\le j \le n+m$, such that 
\[
F_j ^*\omega = \omega _o \quad \text{for }1 \le j \le n \quad \text{and} \quad F_{n+j} ^*\omega = \omega _c \quad \text{for }1 \le j \le m.
\]
(Either $n$ or $m$ can be zero, but not both.)  We also let 
\[
V_j := F_j(\C - 2D_j),
\]
and cutoff functions $\chi _j \in \sC ^{\infty}(X)$ such that 
\[
\chi _j |_{V_j} \equiv 1 \quad \text{and} \quad  {\rm Support}(\chi _j) \subset U_j, \quad 1 \le j \le n+m.
\]

\subsection{Necessity}

Conveniently, necessity of the conditions of Theorem \ref{main} follows rather easily from the special cases of the Euclidean plane and the cylinder.  We therefore reverse the trend set in the special cases, and begin with necessity.

\subsubsection{\sf Uniform separation of interpolation sequences}

\begin{prop}
If $\Gamma$ is an interpolation sequence then $\Gamma$ is uniformly separated in the geodesic distance associated to $\omega$.
\end{prop}

\begin{proof}
Clearly, for each $j$, $\Gamma \cap {U_j}$ is then an interpolation sequence for either the Euclidean case, or the cylindrical case.  It follows that each $\Gamma \cap {U_j}$ is uniformly separated in the geodesic distance for $\omega$.  Since $K$ is compact and $\Gamma$ is a closed discrete subset, the set $\Gamma \cap K$ is finite.  Therefore $\Gamma$ is uniformly separated.
\end{proof}

\subsubsection{\sf Density bound for interpolation sequences}

\begin{prop}
If $\Gamma$ is an interpolation sequence then $D^+_{\vp}(\Gamma) \le1$.  Moreover, if $(X,\omega)$ has no cylindrical ends, then $D^+_{\vp}(\Gamma) < 1$.
\end{prop}

\begin{proof}
For each $j$, $F_j^{-1}(\Gamma \cap {U_j})$ is then an interpolation sequence for either $(\C-D_j,\vp, \omega _o)$ or $(\C^*-D_j, \vp, \omega _c)$.  A moment's thought shows that in our use of Jensen's formula to estimates of the density in the Euclidean and cylindrical settings, we only used our interpolating functions in large disks.  In the course of the proof, the only function we constructed directly (i.e., not from the interpolation hypothesis) was the function interpolating at a point.  Such a function in $\C$ or $\C^*$ can still do the job in $\C-\D_j$.  Thus our method of proof carries over to $\C - D_j$ or $\C^*- D_j$ to get the estimates $D^+_{\vp, j} (F_j^{-1} (\Gamma \cap U_j)) < 1$ for all Euclidean ends, and $D^+_{\vp, j} (F_j^{-1} (\Gamma \cap U_j)) \le 1$ for all cylindrical ends.  
\end{proof}

\begin{s-rmk}
As we mentioned in the introduction, in $\C = \p _1 - \{\infty\}$ with a cylindrical end there is an example of a sequence $\Gamma$ and a weight $\vp$ such that $\tilde D^+_{\vp}(\Gamma) = 1$.  By placing (the tail end of) this sequence in a cylindrical end, we can obtain an example in any asymptotically flat Riemann surface with at least one cylindrical end.
\red
\end{s-rmk}

\subsection{Sufficiency}

As in the special cases of the Euclidean plane and the cylinder, we intend to make use of the $L^2$ Extension Theorem \ref{ot-basic}.  To do so, we need to create the right setting, as we now do.

\subsubsection{\sf Raw densities}

In Definitions \ref{Eucl-density-defn} and \ref{cyl-density-defn}, to define density we replaced $\vp$ with $\vp _r$.  If we use $\vp$ without averaging, the definition can still make sense.  In this case, we call the resulting density the {\it raw density}.  The definition in the Euclidean case is 
\[
\check D ^+_{\vp} (\Gamma) := \inf \left \{ \frac{1}{\alpha} > 0 \ ;\ \Delta \vp \ge \alpha \frac{1}{\pi r^2} \int _{\A ^o_r(z)} \log \frac{r^2}{|\zeta -z|^2} \delta _{\Gamma} (\zeta) \right \}.
\]
In the cylindrical case, the cover density is defined by 
\[
\tilde {\check D}^+_{\vp}(\Gamma) := \check D^+_{\tilde \vp} (\tilde \Gamma).
\]
Finally, in the general case, the raw density 
\[
\check D^+_{\vp} (\Gamma)
\]
of $\Gamma \subset X$ is defined by replacing the density or covered density with their raw counterparts.

\subsubsection{\sf Metric for the (trivial) line bundle associated to $\Gamma$}

Let ${\tilde T} \in \co (X)$ be any holomorphic function such that 
\[
{\rm Ord}({\tilde T}) = \Gamma.
\]
Set 
\[
W_i := F_i (\{ \zeta \in \C\ ;\ {\rm dist}(\zeta , D_i) > r\}), 
\]
where the distance is Euclidean if $\omega |_{U_i}$ is isometric to the Euclidean metric, and cylindrical otherwise.  Define the functions $\lambda ^{\tilde T} _{r,i} :W_i \to \R$ as follows.  If $\omega |_{U_i}$ is isometric to the Euclidean metric, let  
\[
\lambda ^{\tilde T} _{r,i} (z) := \int _{D^o_r(F_i ^{-1}(z))} \log \frac{r^2}{|\zeta - F_i ^{-1}(z)|^2} \log |{\tilde T} \circ F_i ^{-1}(\zeta)|^2 \omega _o (\zeta).
\]
If $\omega |_{U_i}$ is isometric to the cylindrical metric, we choose $x \in \C$ such that the universal covering map $\fp :\C \to \C^*$ maps $x$ to $F_i^{-1}(z)$, and define 
\[
\lambda ^{\tilde T} _{r,i} (z) :=  \int _{D^o_r(x)} \log \frac{r^2}{|\zeta - x|^2} \log |{\tilde T} \circ F_i ^{-1}\circ \fp (\zeta)|^2 \omega _o (\zeta).
\]
Note that if $z \in W_i$ then $\fp (D_r^o(x)) \subset F_i ^{-1}(U_i)$, so that the function $\lambda ^{\tilde T}_{r,i}$ is well-defined on $W_i$ when the latter lies in a cylindrical end.

We then define a function $\lambda _r$ by cutting off the $\lambda ^{\tilde T}_{r,i}$ and dividing by $\pi r^2$:  
\[
\lambda _r := \frac{1}{\pi r^2} \sum _{i=1} ^{n+m} \chi _i \lambda ^{\tilde T}_{r,i}.
\]
Here $\chi _i$ is smooth, takes values in $[0,1]$, is supported in $W_i$, and is identically $1$ on the set 
\[
A_i := F_i (\{ \zeta \in \C\ ;\ {\rm dist}(\zeta , D_i) > r+1\}).
\]
Let 
\[
L := X - \bigcup _{i=1} ^{n+m}A_i.
\]
Then $L$ is compact, and therefore there is a positive constant $M$ such that 
\[
\log |{\tilde T}|^2 - \lambda _r \le M \quad \text{on }L.
\]
On the other hand, the sub-mean value property for subharmonic functions implies that 
\[
\log |{\tilde T}|^2 - \lambda _r \le 0 \quad \text{on }A_i, \quad 1 \le i \le n+m.
\]
Therefore 
\[
\log |{\tilde T}|^2 - \lambda _r \le M \quad \text{on }X.
\]
Letting $T := e^{-M} \tilde T$ (but keeping $\tilde T$ in the definition of $\lambda _r$), we have found functions $T$ and $\lambda _r$ such that 
\[
{\rm Ord}(T)= \Gamma \quad \text{and} \quad |T|^2e^{-\lambda _r} \le 1.
\]

\subsubsection{\sf The semi-strong sufficiency theorem}

Now suppose $\check D^+_{\vp}(\Gamma) < 1$.  If we take $r$ sufficiently large, then there exists $\delta > 0$ such that  
\[
\Delta \vp \ge (1+\delta) \Delta \lambda _r \quad \text{on } A_i, \quad 1 \le i \le n+m.
\]
Since $L \cap \Gamma$ is finite, we also have 
\[
\Delta \lambda _r \le  \frac{1}{r^2} \Omega \quad \text{on }L,
\]
for some positive smooth positive $(1,1)$-form $\Omega$ with compact support on $X$.  

Next, our definition of asymptotic flatness means that ${\rm R}(\omega)$ is compactly supported.  It follows from the curvature hypothesis \eqref{curv-bd-gen} that for $r >> 0$,
\[
\ii \di \dbar \vp + {\rm R}(\omega) \ge (1+\delta ) \Delta \lambda _r.
\]
In view of Theorem \ref{ot-basic}, we have the following result.

\begin{d-thm}\label{wrong-norm-suff-gen}
Let $(X,\omega)$ be an asymptotically flat Riemann surface and let $\vp \in L^1_{\ell oc}(X)$ satisfy the curvature hypothesis 
\[
\Delta \vp + {\rm R}(\omega) \ge m \omega
\]
for some $m>0$.  Let $\Gamma \subset X$ be any closed discrete subset satisfying $D^+_{\vp}(\Gamma) < 1$.  Then for any $f : \Gamma \to \C$ satisfying 
\[
\sum _{\gamma \in \Gamma} \frac{|f(\gamma)|^2e^{-\vp(\gamma)}}{|dT(\gamma)|^2_{\omega} e^{-\lambda _r(\gamma)}} < +\infty
\]
there exists $F \in \sh ^2 (X, e^{-\vp}\omega)$ such that 
\[
F|_{\Gamma} = f \quad \text{and} \quad \int _{X} |F|^2e^{-\vp} \omega _c \le \frac{24\pi}{\delta} \sum _{\gamma \in \Gamma} \frac{|f(\gamma)|^2e^{-\vp(\gamma)}}{|dT(\gamma)|^2_{\omega} e^{-\lambda _r (\gamma)}}.
\]
\end{d-thm}

In view of Propositions \ref{denom-bds-eucl} and \ref{unif-sep-lower-bd-cyl} and the fact $\Gamma$ is uniformly separated if and only if each sequence $\Gamma\cap U_i$ is uniformly separated, we have the following proposition.

\begin{prop}\label{unif-sep-gen}
Let $\Gamma \subset X$ be a closed discrete subset.  Then $\Gamma$ is uniformly separated in the $\omega$-geodesic distance if and only if for each $r >> 0$ there exists $C_r > 0$ such that 
\[
\inf _{\gamma \in \Gamma} |dT(\gamma)|^2_{\omega}e^{-\lambda _r(\gamma)} \ge C_r.
\]
\end{prop}

Combining Propositions \ref{wrong-norm-suff-gen} and \ref{unif-sep-gen} immediately implies the following result.

\begin{d-thm}[Semi-stong sufficiency: general case]\label{semi-strong-suff-general}
Let $(X, \omega)$ be an asymptotically flat finite Riemann surface, and let $\vp \in L^1_{\ell oc}(X)$ be a weight satisfying the curvature hypothesis 
\begin{equation}\label{ss-hyp-gen}
\Delta \vp + {\rm R}(\omega) \ge m \omega
\end{equation}
for some $m>0$.  Assume $\Gamma \subset X$ is uniformly separated with respect to the geodesic distance associated to $\omega$, and that 
\[
\check D^+_{\vp} (\Gamma) < 1.
\]
Then the restriction map $\sr _{\Gamma} :\sh ^2 (X,e^{-\vp} \omega )  \to \ell ^2 (\Gamma , e^{-\vp})$ is surjective.
\end{d-thm}

\subsubsection{\sf Sufficiency: conclusion of the proof of Theorem \ref{main}}

To obtain the sufficiency part of Theorem \ref{main}, we need to replace $\vp$ by some sort of average $\vp _r$ of $\vp$ such that 
\begin{enumerate}
\item[(i)] $\vp _r$ still satisfies \eqref{curv-bd-gen}, and 
\item[(ii)] $\sh ^2(X, e^{-\vp_r}\omega) \cong \sh ^2(X,e^{-\vp}\omega)$ and $\ell ^2(\Gamma, e^{-\vp_r}) \cong \ell ^2(\Gamma,e^{-\vp})$ as topological vector spaces, i.e., the isomorphisms are bounded linear maps.
\end{enumerate}

We already know how to do this in the ends:  in a Euclidean end, we simply replace $\vp$ by its logarithmic average $\vp _r$ defined in \eqref{log-avg-defn}, and in a cylindrical end, we use the covered mean $\mu(\vp)_r$ given in Definition \ref{covered-mean-defn}.

In fact, in the interior it doesn't much matter how we do it; densities are checked only in the ends.  For the sake of deciding on one method, we can cover our compact set $K$ by a finite number of open coordinate charts biholomorphic to disks, and simply replace $\vp$ by its average over a disk of some fixed radius.

After averaging $\vp$ in this way, we multiply the $\vp _{i,r}$ of the end by the cutoff functions $\chi _i$, and multiply the interior averages by any smooth cutoff functions that give a partition of unity on $K$.  (Again, what we do in the interior is not so important.)  If we now sum up all of the cut off averages to form $\tilde \vp_r:= \sum _i \vp _{i,r}$, then clearly 
\[
D^+_{\vp} (\Gamma) = \check D ^+_{\tilde \vp_r} (\Gamma).
\]
The trouble is that in the interior, $\tilde \vp _r$ might not satisfy the curvature hypothesis \eqref{ss-hyp-gen}.  To remedy this, we observe that our underlying Riemann surface $X$ is a compact Riemann surface $Y$ with a finite number of points $x_1,...,x_N$ removed.  Thus there exists a smooth metric of strictly positive curvature for some holomorphic line bundle, say $L \to Y$.  By Kodaira's Embedding Theorem, if $k \in \N$ is sufficiently large then the sections of $L^{\tensor k} \to Y$, embed $Y$ in some projective space.  If we take a basis of sections $\sigma _1,...,\sigma _{N_k} \in H^0(Y, L^{\tensor k})$, we can form the metric 
\[
\psi _k := \log  \sum _{j=1} ^{N_k}  |\sigma_j|^2.
\]
The $\sigma _j$ also define local coordinates on the ambient projective space in which we embedded $Y$, so one can see, after trivializing $L$ near the punctures, that in local projective coordinates, that 
\[
\lim _{z \to x_j} |d\psi _k|^2_{\omega} = 0 \quad \text{and} \quad \lim _{z \to x_j}\frac{\Delta \psi _k}{\omega} = 0.
\]
Thus if we take a cutoff function $\chi $ with compact support in $X$, which is $\equiv 1$ on a sufficiently large compact set containing $K$, and write 
\[
\vp _r := \chi \psi _k + \sum _i \vp _{i,r},
\]
then for sufficiently large $k$, Condition \eqref{ss-hyp-gen} will be satisfied by $\vp _r$.  On the other hand, it is still the case that 
\[
D^+_{\vp} (\Gamma) = \check D ^+_{\vp_r} (\Gamma).
\]
Moreover, it is clear from Proposition \ref{avged-wt} that the needed isomorphisms of the relevant Hilbert spaces holds.  Therefore Theorem \ref{semi-strong-suff-general} implies the sufficiency part of Theorem \ref{main}.

\noi This completes the proof of Theorem \ref{main}.
\qed

\begin{rmk}
Note that unlike the special cases of the Euclidean plane and the cylinder, we did not establish a strong version of sufficiency for the general case; while we were able to eliminate the upper bound in \eqref{curv-bd-gen}, we have retained the lower bound (hence the name `semi-strong').  The main problem is that it is hard to define the density globally on $X$ in such a way that it recovers the covered density in the cylindrical ends.  While it is likely that such a global definition of density exists, for almost any sequence $\Gamma$ the density condition $D^+_{\vp}(\Gamma) < 1$ already implies that the weight $\vp$ satisfies the curvature conditions \eqref{curv-bd-gen}.  Nevertheless, geometrically speaking, it would be interesting to find this global definition of density.
\red
\end{rmk}


\begin{thebibliography}{99}

\bibitem[A-1938]{ahlfors-schwarz} Ahlfors, L., {\it An Extension of Schwarz's Lemma.}  Trans. AMS, Vol. 43, No. 3 (1938), pp. 359 -- 364

%\bibitem[Bd-1971]{beardon} Beardon, A., {\it Inequalities for certain Fuchsian groups.}  Acta Math. 127 1971 221 -- 258.

%\bibitem[B-1996]{bo-otdf} Berndtsson, B., {\it The extension theorem of Ohsawa-Takegoshi and the theorem of Donnelly-Fefferman.} Ann. Inst. Fourier (Grenoble) 46 (1996), no. 4, 1083 -- 1094.

%\bibitem[B-1997]{bo-uniform} Berndtsson, B., {\it Uniform estimates with weights for the $\dbar$ equation.}  The Journal of Geometric Analysis, Volume 7, no. 2 (1997) 195 -- 215

%\bibitem[B-2012]{bo-otq} Berndtsson, B., {\it $L^2$-extension of $\dbar$-closed forms}.  To appear in Illinios Math. J.

\bibitem[BOC-1995]{quimbo} Berndtsson, B.; Ortega Cerd\`a, J., {\it On interpolation and sampling in Hilbert spaces of analytic functions.}  J. Reine Angew. Math. 464 (1995), 109--128.

\bibitem[BL-2010]{bl} Borichev, A.; Lyubarskii, Y., {\it  Riesz bases of reproducing kernels in Fock-type spaces.} J. Inst. Math. Jussieu 9 (2010), no. 3, 449 -- 461.

%\bibitem[BJ-1997]{bj} Bishop, C.; Jones, P., {\it Hausdorff dimension and Kleinian groups.} Acta Math. 179 (1997), no. 1, 1--39.

%\bibitem[B-2013]{blocki-suita} B\l ocki, Z*, {\it  Suita conjecture and the Ohsawa-Takegoshi extension theorem.} Invent. Math. 193 (2013), no. 1, 149 -- 158.

%\bibitem[Bu-1992]{buser} Buser, P., {\it Geometry and spectra of compact Riemann surfaces.} Progress in Mathematics, 106. Birkh\"auser Boston, Inc., Boston, MA, 1992.

%\bibitem[DF-1977]{df} Diederich., K.; Fornaess, J.-E., {\it Pseudoconvex domains: bounded srtictly plurisubharmonic exhaustion functions.} Invent. math., 39, 129 -- 141 (1977).

%\bibitem[D-2002]{diller-green} Diller, J., {\it Green's functions, electric networks, and the geometry of hyperbolic Riemann surfaces.} Illinois J. Math. 45 (2001), no. 2, 453?485.

%\bibitem[DPRS-1985]{sullivan} Dodziuk, J.; Pignataro, T.; Randol, B.; Sullivan, D., {\it Estimating small eigenvalues of Riemann surfaces.} The legacy of Sonya Kovalevskaya (Cambridge, Mass., and Amherst, Mass., 1985), 93?121, Contemp. Math., 64, Amer. Math. Soc., Providence, RI, 1987.

%\bibitem[E-1975]{elen} Elencwajg, G., {\it Pseudo-convexit\'e locale dans les vari\'et\'es kahl\'eriennes.} (French) Ann. Inst. Fourier (Grenoble) 25 (1975), no. 2, xv, 295 -- 314.

%\bibitem[FV-2007]{fv} Forgc\'acs, T.; Varolin, D., {\it Sufficient conditions for interpolation and sampling hypersurfaces in the Bergman ball.} Internat. J. Math. 18 (2007), no. 5, 559Ð584.

%\bibitem[GZ-2012]{gz}Guan, Q.; Zhou, X., {\it Optimal constant problem in the $L^2$ extension theorem.} C. R. Math. Acad. Sci. Paris 350 (2012), no. 15-16, 753--756.

%\bibitem[Ma-1996]{manivel} Manivel, L., {\it Un th\'eor\`eme de prolongement $L^2$ de sections holomorphes d'un fibr\'e hermitien.} Math. Z. 212 (1993), no. 1, 107Ð122. 

%\bibitem[Mc-1996]{jeff-ot} McNeal, J., {\it On large values of $L^2$ holomorphic functions.} Math. Res. Lett. 3 (1996), no. 2, 247 -- 259. 

%\bibitem[MV-2007]{mv1} McNeal, J.; Varolin, D., {\it Analytic inversion of adjunction: $L^2$ extension theorems with gain.} Ann. Inst. Fourier (Grenoble) 57 (2007), no. 3, 703--718.

%\bibitem[N-1967]{nak} Nakai, M., {\it On Evans' kernel.} Pacific J. Math. 22 1967 125--137.

%\bibitem[NS-1970]{sn} Sario, L.; Nakai, M., {\it Classification Theory of Riemann Surfaces}. Springer Verlag (Grundlehren 164) 1970.


%\bibitem[OT-1987]{ot-first} Ohsawa, T.; Takegoshi, K., {\it On the extension of $L^2$ holomorphic functions.} Math. Z. 195 (1987), no. 2, 197 -- 204.

\bibitem[MMO-2003]{mmo} Marco, N.; Massaneda, X.; Ortega Cerd\`a, J., {\it Interpolating and sampling sequences for entire functions.} Geom. Funct. Anal. 13 (2003), no. 4, 862?914.

\bibitem[O-2008]{quim-rs} Ortega Cerd\`a, J., {\it Interpolation and sampling sequences in finite Riemann surfaces.}  Bull. London Math. Soc. 40 (2008) 876 -- 886.

%\bibitem[OSV-2006]{osv} Ortega Cerd\`a, J.; Schuster, A.; Varolin, D., {\it Interpolation and sampling hypersurfaces for the Bargmann-Fock space in higher dimensions.} Math. Ann. 335 (2006), no. 1, 79 -- 107.

\bibitem[OS-1998]{quimseep} Ortega Cerd\`a, J.; Seip, K., {\it Beurling-type density theorems for weighted $L^p$ spaces of entire functions.} Journal d'Analyze, Vol. 75 (1998) 247 -- 266.

\bibitem[PV-2014]{pv} Pingali, V.; Varolin, D., {\it Bargmann-Fock extension from singular hypersurfaces.}  To appear in J. Reine Angew. Math.

%\bibitem[PS-1984]{ps} Pommerenke, C.; Suita, N., {\it Capacities and Bergman kernels for Riemann surfaces and Fuchsian groups}. J. Math. Soc. Japan, Vol. 36, No. 4, (1984). 637 -- 642.

%\bibitem[R-1986]{rudin} Rudin, W., {\it Real and Complex Analysis. 3rd Ed.} McGraw Hill, 1986 (2nd. Printing 1987)

%\bibitem[Sch-1946]{sch}  Schiffer, M.,  {\it The kernel function of an orthonormal system.} Duke Math. J. 13, (1946). 529--540.

%\bibitem[SS-1954]{ss} Schiffer, M.; Spencer, D., {\it Functionals of finite Riemann surfaces}. Princeton University Press, 1954.


\bibitem[SV-2008]{sv1} Schuster, A.; Varolin, D., {\it Interpolation and sampling for generalized Bergman spaces on finite Riemann surfaces.} Rev. Mat. Iberoam. 24 (2008), no. 2, 499 -- 530.

\bibitem[SV-2012]{sv2} Schuster, A.; Varolin, D., {\it Toeplitz operators and Carleson measures on generalized Bargmann-Fock spaces.} Integral Equations Operator Theory 72 (2012), no. 3, 363?392.

%\bibitem[SV-2012]{sv3} Schuster, A.; Varolin, D., {\it New estimates for the minimal $L^2$ solution of $\dbar$ and applications to geometric function theory in weighted Bergman spaces.}  To appear in J. Reine Angew. Math.

\bibitem[Seip-19992]{s1}  Seip, K., {\it Density theorems for sampling and interpolation in the Bargmann-Fock space. I.}  J. Reine Angew. Math. 429  (1992), 91--106.

\bibitem[Seip-1993]{s2}  Seip, K., {\it Beurling type density theorems in the unit disk.}  Invent. Math.  113  (1993),  no. 1, 21--39.

\bibitem[SW-1992]{sw} Seip, K.; Wallst\' en, R. {\it Density theorems for sampling and interpolation in the Bargmann-Fock space. II.} J. Reine Angew. Math.  429  (1992), 107--113.

\bibitem[SS-1961]{ss}  Shapiro, H. S.; Shields, A. L.,  {\it On some interpolation problems for analytic functions.} Amer. J. Math. 83 1961 513 -- 532.

%\bibitem[S-1996]{siu-fujita} Siu, Y.-T. {\it The Fujita conjecture and the extension theorem of Ohsawa-Takegoshi.} Geometric complex analysis (Hayama, 1995), 577 -- 592, World Sci. Publ., River Edge, NJ, 1996.

%\bibitem[S-1998]{siu-pluri} Siu, Y.-T. {\it Invariance of plurigenera.} Invent. Math. 134 (1998), no. 3, 661 -- 673.

%\bibitem[S-2002]{siu-pluri-2}  Siu, Y.-T. {\it Extension of twisted pluricanonical sections with plurisubharmonic weight and invariance of semipositively twisted plurigenera for manifolds not necessarily of general type.} Complex geometry (G\"ottingen, 2000), 223 -- 277, Springer, Berlin, 2002.

%\bibitem[Su-1971]{suita} Suita, N., {\it Capacities and kernels on Riemann surfaces.} Arch. Rational Mech. Anal. 46 (1972), 212 -- 217.

%\bibitem[T-1964]{take} Takeuchi, A., {\it Domaines pseudoconvexes infinis et la m\'etrique riemannienne dans un espace projectif}. (French) J. Math. Soc. Japan 16 1964 159 -- 181.

%\bibitem[T-1975]{tsuji} Tsuji, M., {\it Potential Theory in Modern Function Theory. Second Ed.} Chelsea, 1975

\bibitem[V-2008]{v-tak} Varolin, D., {\it A Takayama-type extension theorem.} Compos. Math. 144 (2008), no. 2, 522--540.

%\bibitem[V-2014]{sbg-1} Varolin, D., {\it Potential theoretic hyperbolicity and $L^2$ extension.  Part I: Stein manifolds.} Preprint, 2014.

%\bibitem[V-2011]{v-book} Varolin, D., {\it Riemann surfaces by way of complex analytic geometry.} Graduate Studies in Mathematics, 125. American Mathematical Society, Providence, RI, 2011.

\bibitem[V-2015]{v-rs2} Varolin, D., {\it Bergman interpolation on finite Riemann surfaces. II. Poincar\'e-hyperbolic case.}  Preprint, 2015

\bigskip

\end{thebibliography}
\end{document}